\documentclass[11pt, reqno]{amsart} 
\usepackage[dvipsnames]{xcolor}
\usepackage[T1]{fontenc}
\usepackage{lmodern}
\usepackage{mathtools, amsfonts, amssymb, mathrsfs, bm, stmaryrd}

\usepackage[utf8]{inputenc}
\usepackage[T1]{fontenc}
\usepackage{graphicx}
\usepackage{tikz}
\usetikzlibrary{shapes, arrows.meta, positioning, calc, backgrounds, lindenmayersystems, decorations.pathreplacing}
\usepackage{pgfplots, pgfplotstable}
\pgfplotsset{compat=1.17}

\usepackage{float}
\usepackage{svg}
\usepackage{wrapfig}
\usepackage{subcaption}
\usepackage{hyperref}
\pgfdeclarelindenmayersystem{cantor set}{
  \rule{F -> FfF}
  \rule{f -> fff}
}
\numberwithin{equation}{section}
\raggedright
\allowdisplaybreaks

\usepackage[letterpaper,hmargin=1.0in,vmargin=1.0in]{geometry}
\parskip 	\smallskipamount

\newcounter{dummy} \numberwithin{dummy}{section}
\newtheorem{corollary}[dummy]{Corollary}
\newtheorem{proposition}[dummy]{Proposition}
\newtheorem{theorem}[dummy]{Theorem}
\newtheorem{definition}[dummy]{Definition}
\newtheorem{lemma}[dummy]{Lemma}

\newtheorem*{remark}{Remark}

\newcommand{\N}{\ensuremath{\mathbb{N}}}
\newcommand{\R}{\ensuremath{\mathbb{R}}}
\newcommand{\Z}{\ensuremath{\mathbb{Z}}}
\newcommand{\Q}{\ensuremath{\mathbb{Q}}}

\newcommand{\PP}{\ensuremath{\mathbb{P}}}

\newcommand{\supp}[1]{\mathrm{supp}_{-\infty}(#1)}

\newcommand{\norm}[1]{\left\lVert#1\right\rVert}

 \newcommand\restrict[2]{
   \left.\kern-\nulldelimiterspace 
   #1
   \littletaller 
   \right|_{#2}%
   }
 \newcommand{\littletaller}{\mathchoice{\vphantom{\big|}}{}{}{}}

 \makeatletter
 \newcommand*{\scaleddelims}[3]{%
   \ensuremath{%
     \mathpalette{\@scaleddelims{#1}{#2}}{#3}%
   }%
 }   
 \newcommand*{\@scaleddelims}[4]{%
   \begingroup
     #3%
     \sbox0{$\m@th#3\vphantom{A}#4$}%
     \setbox2\vbox{\hbox{$\m@th#3#1$}\kern\z@}%
     \setbox4\vbox{\hbox{$\m@th#3#2$}\kern\z@}%
     \setbox6\hbox{$#3\vcenter{}$}%
     \ifx\downharpoonleft#1\relax  
       \let\DelimLeft=L%
     \else\ifx\upharpoonleft#1%
       \let\DelimLeft=L%
     \else\ifx\downharpoonright#1%
       \let\DelimLeft=R%
     \else\ifx\upharpoonright#1%
       \let\DelimLeft=R%
     \fi\fi\fi\fi
     \ifx\downharpoonleft#2\relax
       \let\DelimRight=L%
     \else\ifx\upharpoonleft#2\relax
       \let\DelimRight=L%
     \else\ifx\downharpoonright#2\relax
       \let\DelimRight=R%
     \else\ifx\upharpoonright#2\relax
       \let\DelimRight=R%
     \fi\fi\fi\fi
     \ifx\DelimLeft L%
       \wd2=.6\wd2
     \fi
     \ifx\DelimRight L%
       \wd4=.6\wd4
     \fi
     \ifx\DelimLeft R%
       \sbox2{\kern-.4\wd2\box2}%
     \fi
     \ifx\DelimRight R%
       \sbox4{\kern-.4\wd4\box4}%
     \fi
     \dimen0=\ht0 %
     \advance\dimen0 by -\ht6 %
     \dimen2=\dp0 %
     \advance\dimen2 by \ht6 %
     \ifdim\dimen2>\dimen0 %
       \dimen0=\dimen2 %
     \else
       \dimen0=\dimen0 %
     \fi
     \dimen2=\ht6 %
     \advance\dimen2 by -\dimen0 %
     \dimen0=2\dimen0 %
     \def\DelimCorr{%
       \mskip.5\thinmuskip
       \nonscript\mskip.5\thinmuskip
     }%
     \mathopen{%
       \ifx\DelimLeft R\DelimCorr\fi
       \raisebox{\dimen2}{\resizebox{!}{\dimen0}{\box2}}%
       \ifx\DelimLeft L\DelimCorr\fi
     }%
     \begingroup
       #3#4%
     \endgroup
     \mathclose{%
       \ifx\DelimRight R\DelimCorr\fi
       \raisebox{\dimen2}{\resizebox{!}{\dimen0}{\box4}}%
       \ifx\DelimRight L\DelimCorr\fi
     }%
   \endgroup
 }\makeatother

 \newcommand{\restr}[2]{#1\scaleddelims{\kern-0.5\nulldelimiterspace\upharpoonright}{\vphantom{.}}{_{#2}}}
\newcommand{\diff}{\ensuremath{\operatorname{d}\!}}
\makeatletter
\newcommand{\customlabel}[2]{%
   #2\def\@currentlabel{#2}\label{#1}%
}
\usepackage{comment}
\usepackage{amsfonts, mathrsfs, amsmath}
\usetikzlibrary{positioning}

\title{The KPZ fixed point and Brownian motion share the same null sets}

\author{Pantelis Tassopoulos}
\address{Department of Pure Mathematics and Mathematical Statistics, 
University of Cambridge, Cambridge, United Kingdom}
\email{pkt28@cam.ac.uk}

\author{Sourav Sarkar}
\address{Department of Pure Mathematics and Mathematical Statistics, 
University of Cambridge, Cambridge, United Kingdom}
\email{ss2871@cam.ac.uk}

\begin{document}

\date{}
\begin{abstract}
We show that the increments of the KPZ fixed point started from arbitrary initial data are \emph{mutually} absolutely continuous with respect to Brownian motion with diffusion parameter $2$ on compacts, extending the one-sided Brownian absolute continuity relation of the KPZ fixed point established in \cite{sarkar2021brownian}. 

We also show that additive Brownian motion is absolutely continuous with respect to the centred Airy sheet on compacts, but it is not mutually absolutely continuous globally.

As applications, we show that with probability strictly between zero and one, there exist record times of the KPZ fixed point away from any reference point, obtain a characterisation for the hitting probabilities of the graph of the KPZ fixed point to be positive in terms of a certain thermal capacity in the sense of \cite{watson1978corrigendum, watson1978thermal} and compute essential suprema of Hausdorff dimensions of these random intersections. Finally, we compute essential suprema of Hausdorff dimensions of images of subsets in the plane under the Airy sheet and give a condition for the positivity of their Lebesgue measure in terms of Bessel-Riesz capacity. 
\end{abstract}

\maketitle
\vspace{-0.75cm}
\begin{figure}[H]
    \centering
    \includegraphics[width=0.8\linewidth]{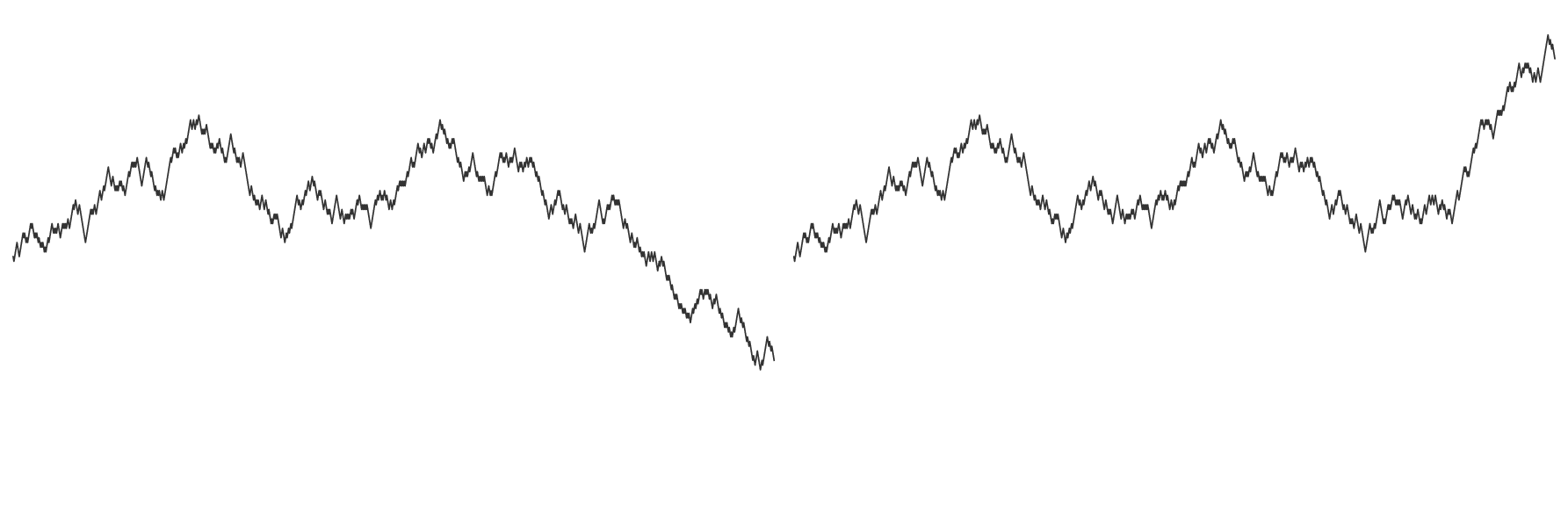}
    \vspace{-0.75cm}
    \caption{\textbf{Left}: the centred KPZ fixed point started from flat initial data. \textbf{Right}: Brownian motion with diffusion parameter $2$.}
    \label{fig: flat bm}
\end{figure}
\tableofcontents

\section{Introduction} 
In 1986, Kardar, Parisi and Zhang \cite{kardar1986dynamic} predicted universal scaling behaviour for many planar random growth processes. Models in the KPZ universality class have a height function $t\mapsto h_t$, $t\ge 0$ which is conjectured to converge at large time and small length scales under the KPZ $1:2:3$ scaling to a universal object called the KPZ fixed point, $t\mapsto \mathfrak{h}_t$, $t\ge 0$. In \cite{quastelkpzfixedpoint2021}, Matetski-Quastel-Remenik constructed the KPZ fixed point as a Markov process in $t$, and they showed that it is a limit of the
height function evolution of the totally asymmetric simple exclusion process with arbitrary initial condition. The natural domain of initial data for the KPZ fixed point is the space of upper semicontinuous functions satisfying a certain sub-parabolic growth condition. The KPZ fixed point at time $t$ started from an admissible initial data $h_0$ can also be expressed in terms of a variational formula with respect to a random metric on space-time $\R^4_\uparrow=\{(x,s;y,t)\in \R^4: s<t\}$, $\mathcal L: \R^4_\uparrow\mapsto\R$, the directed landscape, introduced in \cite{DOV}, as follows
\[\mathfrak{h}_t(y) = \sup_{x\in \R} (h_0(x) + \mathcal{L}(x, 0; y, t))\,.\]

In \cite{sarkar2021brownian}, it was shown that the law of $\mathfrak{h}_t(y)-\mathfrak{h}_t(y_1)$ for $y_1 \le y\le y_2$ is {absolutely continuous} with respect to the law of a Brownian motion starting from $(y_1,0)$ with diffusion parameter $2$ on $[y_1,y_2]$.  We extend the result to \emph{mutual} absolute continuity and show that Brownian motion is absolutely continuous with respect to the law of the KPZ fixed point (up to height shift). This is the main result of this paper, which we now state informally.

\begin{theorem}\label{thm: Absolute continuity KPZ}\label{thm: main informal} Let $t>0$, $-\infty< y_1< y_2<\infty$; then for an arbitrary admissible initial condition $h_0$, the law of $\mathfrak{h}_t(y)-\mathfrak{h}_t(y_1)$ for $y\in [y_1,y_2]$, where $\mathfrak{h}_t$ is the KPZ fixed point at time $t$ started from $h_0$, 
is \emph{mutually absolutely continuous} with respect to the law of a Brownian motion starting from $(y_1,0)$ with diffusion parameter $2$ on $[y_1,y_2]$. 
\end{theorem}

For the precise statement, see Theorem~ \ref{thm: mut abs cont finitary init data}. This gives a very strong comparison between the two laws which is equivalent to saying the KPZ fixed point (up to a height shift) shares the same null events as Brownian motion. 

In Section \ref{sec: top supp airy sheet}, we apply the techniques developed in Section \ref{sec: mut abs cont BM} to the Airy sheet. We prove a non-disjointness result for geodesics in the Airy sheet, Proposition \ref{prop: Airy sheet quadrangle equality} and an absolute continuity result involving additive Brownian motion, Theorem \ref{thm: abs cont airy sheet}. 

In Section \ref{sec: applications}, we discuss applications of Theorem~ \ref{thm: main informal} to the set of record times of the KPZ fixed point in Corollary \ref{cor: record times}. We also characterise when certain hitting probabilities of the graph of the KPZ fixed point are positive in Corollary \ref{cor: thermal cap} and are able to compute the essential suprema of Hausdorff dimensions of these random intersections in Corollaries \ref{cor: int prob haus dim}, \ref{cor: positiveLeb}, \ref{cor: dimh}. These results make use of the full mutual absolute continuity of Theorem~ \ref{thm: mut abs cont finitary init data} (and not just the absolute continuity provided in \cite[Theorem~ 1.2]{sarkar2021brownian}) in an essential way, as a priori, if one drops mutual absolute continuity, one can obtain counterexamples to the statements above. Finally, using Theorem \ref{thm: abs cont airy sheet} and tools from potential theory we study geometric properties of the images of subsets of the plane under the Airy sheet in Corollaries \ref{cor: haus dim airy sheet image} and \ref{cor: bessel-riesz cap pos}. 

\subsection{Related works}

The Brownian nature of models in the KPZ universality class, including the KPZ fixed point, has been a subject of intense research in recent times. Aside from integrable inputs, see for instance \cite{baik1999distribution, quastelkpzfixedpoint2021, liu2019multi} and \cite{johansson2019multi, johansson2017two, johansson2019two}, probabilistic and geometric methods have featured prominently ever since Corwin and Hammond proved in \cite{corwin2014brownian} that the parabolic Airy line ensemble  admits a Brownian Gibbs resampling property (see subsection \ref{subsec: Airy line ensemble}). For a more detailed account of recent developments, one can consult the work of Calvert, Hammond and Hegde \cite{calvert2019brownian} and the references therein.

One version of local Brownianness is to show that the local limits of the
$\text{Airy}_2$ process (the narrow wedge solution to the KPZ fixed point at unit time, i.e. $h_0(0) = 0$ and $h_0(x) = -\infty$, $x\neq 0$) converge in law to a Brownian motion, \cite{hagg2008local}, \cite{cator2015local}, \cite{quastel2013local}. In fact, \cite{quastel2013local} also establishes H\"{o}lder $1/2$-
continuity of the $\text{Airy}_2$ and $\text{Airy}_1$ processes (solution to KPZ fixed point at unit time started from flat, i.e. $h_0\equiv 0$ initial data). The  H\"{o}lder $1/2$- continuity and the locally Brownian nature (in
terms of convergence of the finite dimensional distributions) were established in \cite{quastelkpzfixedpoint2021}. Such H\"{o}lder continuity
results and local limits for certain initial conditions have also been established in \cite{pimentel2018local} and \cite{pimentel2020brownian} (see also \cite{johansson2017two},
\cite{johansson2019two}). A stronger notion of the locally Brownian nature is absolute continuity with respect to Brownian motion
on compact intervals. That the $\text{Airy}_2$ process is Brownian on compacts
was first proved in \cite{corwin2014brownian} using the Brownian Gibbs property; this was considerably considerably strengthened in \cite{dauvergne2024wienerdensitiesairyline}, where boundedness of the Radon Nikodym derivative was established.

For general initial conditions, the picture is less complete. A result providing a more quantitative notion of Brownian regularity, called \emph{patchwork quilt of Brownian fabrics}, was established in Hammond \cite{hammond2019patchwork} and \cite{calvert2019brownian}. Roughly the result states that the KPZ fixed point $\mathfrak{h}_t(\cdot)$ on a unit interval is the result of `stitching' a random number of profiles (or patches), where each profile is absolutely continuous with respect to a Brownian motion with Radon-Nikodym derivative in $L^p$ for all $p<3$. The authors conjectured (Conjecture $1.3$ in \cite{hammond2019patchwork}) that one can dispense with these random patches and establish $L^p$ estimates for all $p>1$ for the Radon-Nikodym derivative, a problem which remains open. 

By different means, the authors in \cite{sarkar2021brownian} proved absolute continuity of the KPZ fixed point with respect to Brownian motion on compacts for general initial conditions, \cite[Theorem~ 1.2]{sarkar2021brownian}. In \cite{tassopoulos2025inhomogeneousbrownian, tassopoulos2025quantitativebrownianregularitykpz}, we strengthened the above comparison by obtaining an explicit functional relationship between the law of the increments of the KPZ fixed point started from arbitrary initial data and Brownian motion on compacts. Whether the KPZ fixed point is \emph{mutually} absolutely continuous with respect to Brownian motion on compacts, still remained open.

Our main result in Theorem~ \ref{thm: mut abs cont finitary init data} of this paper settles the question of mutual absolute continuity of the KPZ fixed point started from arbitrary initial data against Brownian motion on compacts, by answering it in the affirmative. It crucially leverages the fact that the directed landscape at unit time, the Airy sheet, can be fully recovered as a deterministic function of the Airy line ensemble, \cite[Theorem~ 1.21]{dauvergne2022scalinglimitlongestincreasing}. In particular, we use a new coupling between the Airy sheet and the Airy line ensemble which extends the coupling in \cite{DOV}.

In Section \ref{sec: top supp airy sheet}, we apply the techniques developed in Section \ref{sec: mut abs cont BM} to the Airy sheet. In particular, we prove a non-disjointness result for the geodesics in the directed landscape in Proposition \ref{prop: Airy sheet quadrangle equality}. Moreover, we prove that the additive Brownian motion is absolutely continuous with respect to the Airy sheet (up to centering) on compact subsets of $\R^2$, Theorem \ref{thm: abs cont airy sheet}.

We then discuss applications of the tools developed in Sections \ref{sec: mut abs cont BM} and \ref{sec: top supp airy sheet}. We revisit a certain notion of thermal
capacity from \cite{watson1978corrigendum, watson1978thermal}, \cite{khoshnevisan2015brownianmotionthermal} and use Theorem \ref{thm: mut abs cont finitary init data} to characterise when certain hitting probabilities of the KPZ fixed point are positive.  Finally, using Theorem \ref{thm: abs cont airy sheet}, we compute essential suprema of their Hausdorff dimensions of images of compact subsets in the plane under the Airy sheet and give a condition for them to have positive Lebesgue measure using potential theory for additive Brownian motion, see for example \cite{khoshnevisan1998browniansheetimagesbesselriesz}.

\vspace{-0.1cm}
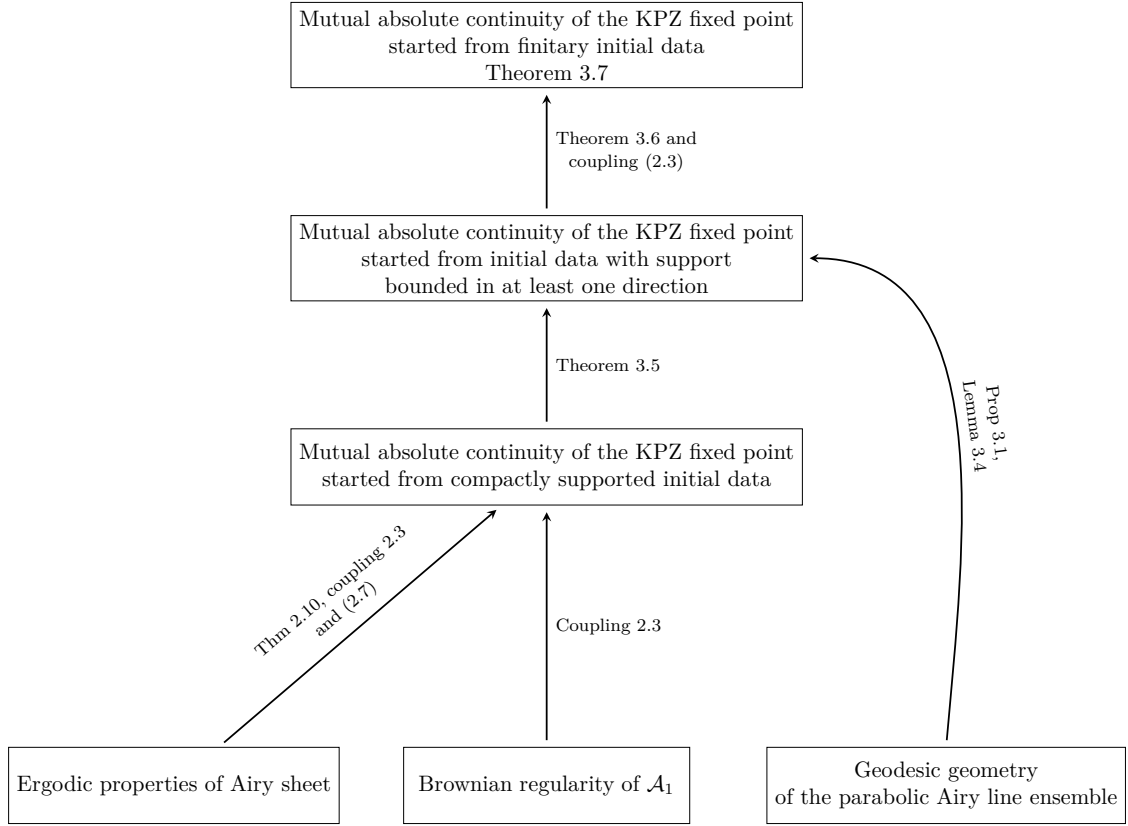
\begin{figure}
\centering

{\NoHyper\resizebox{0.9\textwidth}{!}{\begin{tikzpicture}[
    node distance=2.5cm and 2.5cm,
    startstop/.style={rectangle, minimum width=4.5cm, minimum height=1.2cm, text centered, draw=black, align=center, font=\small},
    process/.style={rectangle, minimum width=4.5cm, minimum height=1.2cm, text centered, draw=black, align=center, font=\small},
    arrow/.style={thick,->,>=stealth,shorten >=3pt,shorten <=3pt}
  ]

  \node (start) [startstop] 
    {Mutual absolute continuity of the KPZ fixed point\\ started from finitary initial data\\ Theorem~ \ref{thm: mut abs cont finitary init data}};

  \node (finitary reg) [process, below=2cm of start] 
    {Mutual absolute continuity of the KPZ fixed point\\ started from initial data with support \\ bounded in at least one direction};

  \node (comp reg) [process, below=2cm of finitary reg] 
    {Mutual absolute continuity of the KPZ fixed point\\ started from compactly supported initial data};

  \node (above) [process, below=3.8cm of comp reg] 
    {Brownian regularity of $\mathcal{A}_1$};

  \node (extra2) [process, left=1cm of above] 
    {Ergodic properties of Airy sheet};

  \node (extra) [process, right=1.2cm of above] 
    {Geodesic geometry\\ of the parabolic Airy line ensemble};

  
  \draw [arrow] (finitary reg) -- node[midway, right, font=\scriptsize, align=center] 
    {Theorem~ \ref{thm: mut abs cont unbounded supp init data} and \\ coupling \eqref{eq: airy sheet coupling full space}} (start); 
    
  \draw [arrow] (comp reg) -- node[midway, right, font=\scriptsize] 
    {Theorem~ \ref{thm: mutual abs cont}} (finitary reg); 

  \draw [arrow] (above) -- node[midway, right, font=\scriptsize] 
    {Coupling \ref{eq: airy sheet coupling full space}} (comp reg); 

  \draw [arrow] (extra2) -- node[midway, sloped, transform shape, anchor=south, font=\scriptsize, yshift=2pt, align=center] 
    {Thm \ref{thm: erg airy line ensemble}, coupling \ref{eq: airy sheet coupling full space}\\ and \eqref{eq: ergod airy sheet}} (comp reg); 

  \draw [arrow] (extra.north) to[out=85, in=0, looseness=1.0] 
    node[pos=0.5, sloped, transform shape, anchor=south, font=\scriptsize, yshift=2pt, align=center] 
    {Prop \ref{prop: separation of initial data},\\ Lemma \ref{lemma: infinite lpp cont at zero}} (finitary reg.east); 

\end{tikzpicture}}\endNoHyper}

\caption{Flowchart of main steps in the proof of Theorem~ \ref{thm: mut abs cont finitary init data}.}
  \label{fig: flowchart mut abs cont}
\end{figure}

\subsection{Organization of the paper}\label{sec: prelims}

First, in Section \ref{sec: prelims} we provide necessary background material including properties of last passage percolation, the Pitman transform and facts about Brownian bridges, ergodic properties of the Airy line ensemble, the Brownian regularity of the parabolic Airy$_2$ process, the symmetries of the Airy sheet as well as couplings thereof to the Airy line ensemble, ending the section with some background and setup for the KPZ fixed point.

In Section \ref{sec: mut abs cont BM}, we prove that for arbitrary initial data one obtains mutual absolute continuity of the laws of the spatial increments of the KPZ fixed point against rate two Brownian motion on compacts; the arguments leverage the construction of the entire Airy sheet as a deterministic function of the Airy line ensemble. Figure \ref{fig: flowchart mut abs cont} shows the key steps of the proof of the main result, Theorem~ \ref{thm: mut abs cont finitary init data}.

In Section \ref{sec: top supp airy sheet}, we prove a non-disjointness result for the geodesics in the Airy sheet, Proposition \ref{prop: Airy sheet quadrangle equality} and that the additive Brownian motion is absolutely continuous with respect to the (centred) Airy sheet on compacts, Theorem \ref{thm: abs cont airy sheet}, but is not mutually absolutely continuous, Proposition \ref{prop: not mut abs cont airy sheet additive brownian motion}.

We then discuss applications of the mutual absolute continuity result, Theorem \ref{thm: mut abs cont finitary init data}, and the absolute continuity result, Theorem \ref{thm: abs cont airy sheet}, in Section \ref{sec: applications}. These pertain to record times of the KPZ fixed point, certain hitting probabilities for the graph of the KPZ fixed point in terms of a certain parabolic capacity and computations of $L^\infty$-norms of Hausdorff dimension of these (random) intersections. Finally, we use Theorem \ref{thm: abs cont airy sheet} and tools from potential theory to study geometric properties of the images of subsets of the plane under the Airy sheet, Corollary \ref{cor: Airy sheet top supp}, including computing the essential suprema of the Hausdorff dimensions of these random sets.

Finally, in the Appendix, Section \ref{sec: appendix}, we prove a non-degeneracy result for geodesic jump times in the Airy line ensemble which we use in  Section \ref{sec: mut abs cont BM}.

\subsection{Notation} Let $\mathscr{C}([a,b];\R^d)$ denote the space of all continuous functions $f:[a,b]\mapsto \R^d$, $d\ge 1$ and $\mathscr{C}_0([a,b]; \R^d)$ denote the space of all $f\in \mathscr{C}([a,b];\R^d)$ with $f(a)=0$. 

We say that a Brownian motion or a Brownian bridge has {rate (or diffusion parameter) $v$} if its quadratic variation in an interval $[s,t]$ is equal to $v(t-s)$. From now on, all Brownian motions/bridges are rate two unless stated otherwise. Moreover, we denote by $\mu$ the Wiener measure associated to a rate two Brownian motion starting from the origin on $\mathscr{C}([0, \infty);\R)$ (or in a slight abuse of notation starting from any other point on the line).

Finally, for a random variable $Y$ on some probability space $(\Omega, \mathscr{F}, \PP)$, we will sometimes denote a version of the regular conditional distribution of $\PP$ given $Y$ (whenever the latter exists, see \cite[Theorem 8.5]{kallenberg2021foundations} for sufficient conditions which will suffice in the present case), by $\PP_Y(\cdot) \equiv \PP(\cdot | \sigma(Y))$. For sigma algebras $\mathcal{A}, \mathcal{B}$ on some set $\Omega$, we denote the minimal sigma algebra containing both $\mathcal{A}$ and $\mathcal{B}$ by $\mathcal{A}\lor \mathcal{B}$.

\section{Preliminaries}\label{sec: prelim}

We first recall the definition of absolute continuity of measures on a measurable space $(\Omega, \Sigma)$. 

\begin{definition}[Absolute continuity]\label{def: abs cont}
    Let $\mu, \nu$ denote measures on $(\Omega, \Sigma)$. Then, we say $\mu$ is \emph{absolutely continuous} with respect to $\nu$, written as $\mu \ll \nu$, if for all $A\in \Sigma$ such that $\nu(A) = 0$, $\mu(A) = 0$. We say two measures $\mu$ and $\nu$ are \emph{mutually absolutely continuous} if both $\mu \ll \nu$ and $\nu \ll \mu$ are satisfied.
\end{definition}

In what follows, a \textit{random line ensemble} is a random variable taking values in an indexed (at most countably infinite) family of continuous paths defined on a common subset of $\R$. 

\subsection{Last passage percolation} We begin with the collection of some preliminary facts regarding last passage percolation (LPP). For more details, see \cite[Section 2]{DOV}.

Formally, let \(I\subset \mathbb{Z}\) be a possibly finite index set and define the space \(\mathscr{C}^I\) of sequences of continuous functions with real domains, that is, the space of maps
$f: \mathbb{R}\times I\to \mathbb{R}\quad (x,i)\mapsto f_i(x)$.

\begin{definition}[Path] Let \(x\leq y \in \mathbb{R}\), and \(m\leq \ell \in \mathbb{Z}\) respectively. A \emph{path}, from \((x,\ell)\) to \((y,m)\) is a non-increasing  function \(\pi: [x,y] \to \mathbb{N}\) which is cadlag on \((x,y)\) and takes the values \(\pi(x)= \ell\) and \(\pi(y)= m\).
\label{def: path}
\end{definition}

This also leads one to naturally define a derived quantity, namely the \textit{last passage value}.

\begin{definition}[Length]\label{def: length} Let \(x\leq y \in \mathbb{R}\) and \(m < k\in\mathbb{Z}\). For each \(m\leq i <k\), let \(t_{k-i}\) denote the jump of the path \(\pi\), on an ensemble \((f_i)_{i\in I}\), from \(f_{i+1}\) to \(f_{i}\). Then the length of \(\pi\) is defined as
\[
\ell(\pi) = f_m(y)-f_m(t_{k-m}) + \displaystyle\sum_{i = 1}^{\ell-m-1}(f_{k-i}(t_{i+1})-f_{k-i}(t_{i}))+f_{k}(t_{1})-f_{k}(x)\,.
\]
\end{definition}

\begin{definition}[Last passage value]\label{def: last passage}
    With \(x\leq y, m<k\) as before and \(f\in \mathscr{C}^I\), define the \emph{last passage value} of \(f\) from \((x,k)\) to \((y,m)\) as
    \begin{equation*}
    f[(x,k)\to(y,m)] \stackrel{\mathrm{def}}{=}\displaystyle \sup_{\pi}\ell(\pi)\,,
    \end{equation*}
where the supremum is over precisely the paths \(\pi\) from \((x,k)\) to \((y,m)\).
\end{definition}
\begin{remark}
    Any path \(\pi\) from \((x,k)\) to \((y,m)\) such that its length is equal to its last passage value is called a \emph{geodesic}. 
\end{remark}

Last passage percolation enjoys the following \emph{metric composition law}, Lemma 3.2 in DOV \cite{DOV}.

\begin{lemma}[Metric composition law]\label{Lemma: Metric Composition}
    Let \(x\leq y \in \mathbb{R}\), \(m < \ell\in\mathbb{Z}\) and \(f\in \mathscr{C}^I\). If \(k\in \{m, \dots, \ell\}\), then we have
    \[
    f[(x,\ell)\to(y,m)] = \displaystyle \sup_{z\in[x,y]}(f[(x,\ell)\to(z,k)]+f[(z,k)\to(y,m)])\,.
    \]
    Furthermore, for any \(z\in [x,y]\), 
    \begin{equation}\label{eq: composition}
    f[(x,\ell)\to(y,m)] = \displaystyle \sup_{k\in \{m, \dots, \ell\}}(f[(x,\ell)\to(z,k)]+f[(z,k)\to(y,m)])
    \end{equation}
\end{lemma}

\subsection{Pitman transform}\label{subsec: pitman trans}

Recall that with $f=(f_1,f_2)$ where $f_i:[0,\infty)\mapsto \R$ for $i=1,2$, for $f\in \mathscr{C}^2([0,\infty);\R)$, we define $\mathrm \mathrm{W}f=(\mathrm{W} f_1,\mathrm{W} f_2)\in \mathscr{C}^2([0,\infty);\R)$, the \emph{Pitman transform} of $f$ as follows. For $x<y\in [0,\infty)$, define the maximal gap size
\[G(f_1,f_2)(x,y)\equiv\max\left(\max_{s\in [x,y]}\big(f_2(s)-f_1(s)\big)\,,\,0\right)\,.\]
Then define
\begin{equation}\label{eq: pitmantrans}
\mathrm{W} f_1(t)=f_1(t)+G(f_1,f_2)(0,t)\,, \mathrm{W} f_2(t)=f_2(t)-G(f_1,f_2)(0,t)\,,
\end{equation}
for all $t\in [0,\infty)$.

One can express the top line of the Pitman transform (also known as the \emph{Skorokhod reflection} of $f_1$ against $f_2$, see \cite[Lemma 2.1]{revuz2013continuous}) in terms of last passage values. It is easy to see that (see for example \cite[Section 2.1]{sarkar2021brownian}) for all \(t\in[0,\infty)\),
    \[
    Wf_1(t) = \displaystyle \max_{i=1,2}\{f_i(0) + f [(0, i) \to (t, 1)]\} \,.
    \]

For continuous functions $f_1, \ldots, f_n$, starting with $f_n$, reflecting $f_{n-1}$ off of $f_n$ to give $W(f_{n-1, f_n})_1$ and so on gives at the final stage
\[
\max_{1\le i \le n}\{f_i(0) + f [(0, i) \to (t, 1)]\,,t\ge 0\,.
\]
The values $f_i(0)$ will also be called the \emph{boundary data}. For more details, see the construction involving \emph{inhomogeneous Brownian last passage values} \cite[Section 2]{tassopoulos2025inhomogeneousbrownian}.

\subsection{Brownian bridge properties}\label{sec: bbridge}

Here we put together a few standard facts and basic lemmas on Brownian bridges, that will be needed in the later sections.

We will make frequent use of the following standard lemma (stated informally below) comparing a Brownian bridge away from its right endpoint to a Brownian motion. For a more precise statement and proof, please see \cite[Lemma 3.9]{tassopoulos2025quantitativebrownianregularitykpz}.

\begin{lemma}\label{lemma: bb comparison lemma}
    Fix $0<x<y$, $m\in \N$ and let $W(\cdot)$ be an $m$-dimensional Brownian bridge on $[0, y]$ with endpoints $\underline{0}, \underline{a}\in \R^m$. Then the law of $W(\cdot)$ restricted to $[0,x]$ is mutually absolutely continuous with respect to that of an $m$-dimensional Brownian motion on $[0, x]$ starting from $(0, \underline{0})$.
\end{lemma}

We record a standard result about the topological support of Brownian bridges on path space.

\begin{lemma}\label{lemma: support bb}
   Let $f: [0,1] \to \R$ be a continuous function with $f(0) = f(1) = 0$. Then with $W$ a two-sided rate two Brownian bridge vanishing at both endpoints, and any open set $U$ (with respect to the topology of uniform convergence on $[0,1]$) that contains $f$, 
    \[
    \PP(W(\cdot) \in U) > 0\,.
    \]
\end{lemma}
\begin{remark}
    In conjunction with Lemma \ref{lemma: bb comparison lemma}, Lemma \ref{lemma: support bb} yields the corresponding result for Brownian motion.
\end{remark}

We now state informally a decomposition result for Brownian bridges, which is similar in spirit to the L\'{e}vy-Ciesielski construction of Brownian motion.

\begin{lemma}[Lemma 2.8 in \cite{corwin2014brownian}]\label{lemma: bb decomp}
Fix $j\in \N$, $T>0$ and consider a sequence of times $0=t_0<t_1<\cdots < t_j=T$. There exists a sequence of independent centered Gaussian random variables $\{N_i\}_{i=1}^{j-1}$, interpolation functions $I_i$, $1\le i \le j$ and a sequence of independent Brownian bridges $\{B_i\}_{i=1}^j$ such that $B_i:[0,t_i-t_{i-1}]\rightarrow \R$ vanishes at both endpoints such that the random function $B:[0,T]\rightarrow \R$,
\begin{equation*}
B(s) = \sum_{i=1}^{m(s)+1} I_i(s) +   B_{m(s)+1}(s-t_{m(s)})\, ,
\end{equation*}
with $m(s) = \max \big\{i:t_i<s \big\}$
is equal in law to a Brownian bridge $B'$ on $[0,T]$ with arbitrary endpoints $B(0)$ and $B(T)$.
\end{lemma}

We end this subsection with a key monotonicity lemma for Brownian bridges.

\begin{lemma}(Monotonic coupling)\label{lemma: bridge monotonicity}
Let $[s, t], J$ be closed intervals in $\R$ with $J \subseteq [s, t]$, let $\underline{x}^1 \le \underline{x}^2, \underline{y}^1 \le \underline{y}^2 \in \R^k_>$ where $\le$ is the coordinate-wise partial order, and let $g_1, g_2$ be two bounded Borel measurable functions from $[s,t] \to \R \cup \{-\infty\}$ such that $g_1(x)\le g_2(x)$ for all $x \in [s,t]$. For $i = 1, 2$, let $B^i$ be a $k$-tuple of Brownian bridges from $(s, \underline{x}^i)$ to $(t, \underline{y}^i)$, conditioned to avoid each other and $g_i$. Then there exists a coupling such that $B^1_j(r) \le B^2_j(r)$ for all $r \in [s, t], j \in \llbracket 1, k\rrbracket$.
\end{lemma}

For a sketch of a proof, see the proof of Lemmas $2.6$ and $2.7$ in \cite{corwin2014brownian}. For a more complete argument, see the proof of Lemma $2.15$ in \cite{dimitrov2021characterisation}. 

\subsection{Airy line ensemble and the Brownian Gibbs property}\label{subsec: Airy line ensemble}
The Airy line ensemble is a non-intersecting random
sequence of continuous functions \(\mathcal{A} = (\mathcal{A}_1, \mathcal{A}_2, \dots)\) (see Theorem~ 2.1 in \cite{DOV}), such that $\mathcal{A}_1 > \mathcal{A}_2 > \cdots$. It was introduced by Pr\"{a}hofer and Spohn \cite{prahofer2002scale} in the version \((\mathcal{A}^{\mathrm{stat}}_i)_{i\in \N}\stackrel{\mathrm{def}}{=}(\mathcal{A}_i(\cdot)+(\cdot)^2)_{i\in \N}\), which is stationary
in time, see also \cite{corwin2014brownian} and \cite{corwin2014ergodicityairylineensemble}. We will thus call it the \emph{stationary Airy line ensemble}. The top line $\mathcal{A}_1$ is known as the parabolic $\text{Airy}_2$ process that appears as the
limiting spatial fluctuation of random growth models starting from a single point.

We now recall the Brownian Gibbs resampling property enjoyed by the Airy line ensemble, first established in \cite{corwin2014brownian}. Informally, it states that for $a< b$, $k\in \N$, the law of the Airy line ensemble restricted to $\{1,2,\cdots,k\}\times(a,b)$, ${\mathcal{A}}|_{\{1,2,\cdots,k\}\times(a,b)}$, conditionally on all the data generated by the Airy line ensemble outside of this region, $\mathscr{F}^\mathrm{ext}_{\llbracket 1,k\rrbracket \times (a,b)}\equiv \sigma(\{\mathcal{A}_i(x): (i,x)\notin \llbracket 1,k\rrbracket \times (a,b)\})$, is given by non-intersecting Brownian bridges with entry data $\underline{x} = (\mathcal{A}_i(a))_{1\leq i\leq k}$, $\underline{y} = (\mathcal{A}_i(b))_{1\leq i\leq k}$ and also conditioned to stay above $f = \mathcal{A}_{k+1}$ on $(a,b)$. For the precise statement, see \cite[Section 2]{tassopoulos2025quantitativebrownianregularitykpz}.

We record the following theorem which shows that the stationary Airy line ensemble is \textit{ergodic}.

\begin{theorem}(\cite[Theorem~ 1.6]{corwin2014ergodicityairylineensemble})\label{thm: erg airy line ensemble}
    The stationary Airy line ensemble is ergodic with respect to horizontal shifts. 
\end{theorem}

We end this subsection with a statement regarding the strong comparison the top line of the Airy line ensemble, the parabolic Airy$_2$ process, enjoys against Brownian motion on compacts. 

\begin{theorem}[Theorem 1.1. in \cite{dauvergne2024wienerdensitiesairyline}]\label{thm: mut abs cont airy bm}
    Fix $J\subseteq \R$ a bounded interval. Then law $\nu$ of the increments of the parabolic Airy$_2$ process 
    $\mathcal{A}_1(\cdot+\inf J) - \mathcal{A}_1(\inf J)$
    on paths $\mathscr{C}_0([0, \sup J -\inf J] ;\R)$ is mutually absolutely continuous with respect to the law of $\mu$, a rate 2 Brownian motion on $[0, \sup J -\inf J ]$.
\end{theorem}

\subsection{The Airy sheet and the directed landscape}\label{sec: airy sheet}
The standard Airy sheet $\mathcal{S}:\R^2\mapsto \R$ is a random continuous function defined in terms of the Airy line ensemble such that $\mathcal{S}(0,\cdot)=\mathcal{A}_1(\cdot)$, first constructed in \cite{DOV}. The Airy sheet of scale $s$ is defined by
\[\mathcal{S}_s(x,y)=\mathcal{S}(x/s^2,y/s^2)\,,\]
for any $s>0$. 

We collect some important properties of the Airy line ensemble and the Airy sheet, which will prove useful later. Recall that $\mathcal{A}$ is the parabolic Airy line ensemble and that $\mathcal{S}$ is the Airy sheet. 

In \cite{dauvergne2022scalinglimitlongestincreasing}, the Airy sheet was constructed as a deterministic function of the Airy sheet on the entire plane extending the coupling on the half-plane used in \cite{DOV}. As a by-product, the Airy sheet $\mathcal{S}(\cdot, \cdot) $ can be coupled with the (parabolic) Airy line ensemble $\mathcal{A}$ so that $\mathcal{S}(0,\cdot)=\mathcal{A}_1(\cdot)$ and almost surely for all $(x,y,z)\in \Q\setminus\{0\}\times \Q^2$, there exists a random integer $K_{x,y,z}$ such that for all $k\ge K_{x,y,z}$, (recalling the definition for last passage values, Definition \ref{def: last passage})
\begin{align}\label{eq: airy sheet coupling full space}
\mathcal{S}(x,z)-\mathcal{S}(x,y) &= \mathcal{A}[x_k\to (z,1)]-\mathcal{A}[x_k\to (y,1)]\,,\quad x > 0\nonumber\\
\mathcal{S}(x,z)-\mathcal{S}(x,y) &= \tilde{\mathcal{A}}[(-x)_k\to (-z,1)]-\tilde{\mathcal{A}}[(-x)_k\to (-y,1)]\,,\quad x < 0\,,
\end{align}
where $x_k=(-\sqrt{k/2x},k)$, $x> 0$ and $\tilde{\mathcal{A}}(\cdot) = \mathcal{A}(-\cdot)$.

Aside from this coupling, we will also need some global shape estimates for the Airy sheet. In particular, the Airy sheets satisfy almost sure pointwise bounds 
\begin{equation}\label{eq: airy shape bnds}
|\mathcal{S}(x,y)+(x-y)^2|\le \mathfrak{C}+c\log^{2/3}(2+|x|+|y|)\,,\qquad \text{ for all } x,y\in \R 
\end{equation}
for some universal constant $c>0$ and some $\mathfrak{C}$ satisfying $\mathbb{E}[a^{\mathfrak{C}^{3/2}}]<\infty$ for some $a>1$, \cite{sarkarthrehalves2022}. 

Finally, we recall some properties of the Airy sheet, from Section 9 in \cite{DOV} and Section 14 in \cite{dauvergne2022scalinglimitlongestincreasing}. More precisely, we have almost surely, that as a random continuous function in $\R^2$, the Airy sheet is translation invariant, that is, for any $c\in \R$, $\mathcal S(\cdot+c, \cdot+c)\stackrel{\mathrm{d}}{=} \mathcal  S(\cdot, \cdot)$ and  
\begin{equation}\label{eq: Airy sheet symmetry}
\mathcal{S}(x, y) \stackrel{\mathrm{d}}{=} \mathcal{S}(-y, -x)\,, \quad \mathcal{S}(x, y) \stackrel{\mathrm{d}}{=} \mathcal{S}(-x, -y) \quad \mathrm{and} \quad \mathcal{S}(x, y) \stackrel{\mathrm{d}}{=} \mathcal{S}(x, y + c) + 2c(y-x) + c^2.
\end{equation}
Moreover, almost surely for all $x\le x',\; y\le y' $
\begin{equation}\label{eq: Airy sheet monoton}
    \mathcal{S}(x,y)+\mathcal{S}(x',y')\ge \mathcal{S}(x,y')+\mathcal{S}(x',y).
\end{equation}

Since $\mathcal S(x,y+\cdot)\stackrel{\mathrm{d}}{=}\mathcal S(0,y-x+\cdot)=\mathcal A_1(y-x+\cdot)$ by the translation invariance and the coupling above,  we now have, by Theorem~ \ref{thm: erg airy line ensemble} and the pointwise ergodic theorem for all $x, y \in \R^2$ almost surely,
\begin{align}\label{eq: ergod airy sheet}
    \mathcal{S}(x, y)+(x-y)^2 &= \mathbb{E}[\mathcal{A}_1(0)]+\lim_{m\to \infty}\frac{1}{m}\sum_{m=1}^\infty (\mathcal{S}(x, y)-\mathcal{S}(x, y+m)+(x-y)^2 -(x-y-m)^2)\nonumber\\
    &= \mathbb{E}[\mathcal{A}_1(0)]+\lim_{m\to \infty}\frac{1}{m}\sum_{m=1}^\infty (\mathcal{S}(x, y)-\mathcal{S}(x, y-m)+(x-y)^2 -(x-y+m)^2)\,.
\end{align}

Now we introduce some geodesic geometry on the Airy line ensemble. For $x\leq y\in \R$ and $\ell\in \N$, we shall denote the rightmost geodesic between $(x,\ell)$ and $(y,1)$ in the Airy line ensemble $\mathcal{A}$ by $\pi[(x,\ell)\to y]$ (see Section 2 of \cite{sarkar2021brownian}). Next we recall the definition of infinite geodesics.

\begin{definition}\label{def: semi-inf geo} For any $x\in\R^+$ and $y\in \R$ with $x_k=(-\sqrt{k/2x},k)$, we define the geodesic $\pi[x\to y]$ as the almost sure pointwise limit of $\pi[x_k\to y]$ as $k\to \infty$, whenever the limit exists. We define the length of the geodesic $\pi[x\to y]$ as $\mathcal{S}(x,y)$. 
\end{definition}
\begin{remark}
    The fact that these limits exist almost surely for all $x,y$ in a countable dense set of $\R^+\times \R^2$ is the content of \cite[Lemma 3.4]{sarkar2021brownian}.
\end{remark}

More generally, the directed landscape $\mathcal L: \R^4_\uparrow=\{(x,s;y,t)\in \R^4: s<t\}\mapsto\R$ is a random continuous function satisfying
the metric composition law
\begin{equation}\label{eq: metric comp}
    \mathcal{L}(x,r;y,t)=\sup_{z\in \R}(\mathcal{L}(x,r;z,s)+\mathcal{L}(z,s;y,t))\,,
\end{equation}
for all $(x,r,y,t)\in \R^4_\uparrow$ and all $s\in (r,t)$; and with the property that $\mathcal{L}(\cdot, t;\cdot, t + s^3)$ are independent Airy sheets of scale $s$ for any set of
disjoint time intervals $(t,t+s^3)$. The directed landscape $\mathcal L(x,s;y,t) $ can be thought of as a (random) metric between space-times points $(x, s)$ and $(y, t)$.

\subsection{The KPZ fixed point}

We briefly discuss the state space of admissible initial data for the KZP fixed point, namely the space of upper semicontinuous functions from $\mathbb{R}$ to $\mathbb{R}\cup\{\pm \infty\}$, $\mathrm{UC}$, (see Section 3 of \cite{quastelkpzfixedpoint2021} and the Appendix in \cite{virag2025actions} for details) with sub-parabolic growth at infinity. 

Next, we need an appropriate definition of `support' compatible with the `max-plus' nature of the directed landscape.

\begin{definition}(max-plus support)\label{def: max support}
    Let $f:\R\to \R\cup \{-\infty\}$ be a Borel function. We define the \emph{max-plus} support of $f$ to be the set
    \[
    \mathrm{supp}_{-\infty}(f) : = \{x\in \R: f(x)\neq -\infty\}\,.
    \]
\end{definition}

We now define the class of compactly supported upper-semicontinuous functions in the `max-plus' sense.

\begin{definition}\label{def: comp supp UC}
    Denote the class of compactly supported, upper-semi continuous functions on the line,
    \[
    \mathrm{UC}_c \stackrel{\mathrm{def}}{=}\{f \in \mathrm{UC}: \supp{f} \text{ is bounded }\}\,.
    \]
\end{definition}

Now, for $I\subseteq \R$, we introduce the following notation for the `restriction operator' acting on $f\in \mathrm{UC}$ by pointwise multiplication, truncating its `max-plus' support to $I$, (with the convention $0\cdot (- \infty) = -\infty)$)
\begin{equation}\label{eq: max-plus indicator}
    \delta_{I}(x) = \begin{cases}
    &1 \qquad \,\ \mbox{ for }x\in I\,,\\
    &-\infty\,\quad \mbox{ for } x\in  \R\setminus I\,.
    \end{cases}
\end{equation}

We now make explicit the parabolic growth condition for initial data in $\mathrm{UC}$. For $t>0$, recall the definition of \textit{$t$-finitary} initial data.
\begin{definition}($t$-finitary initial data)\label{def: finitary}
For $t> 0$, we denote by $\mathscr{I}_t$ the set of locally bounded upper semicontinuous $h_0\in \mathrm{UC}$ satisfying the sub-parabolic growth condition
    \[
    \displaystyle\lim_{|x|\to \infty}\frac{h_0(x) - x^2/t}{|x|} = -\infty\,.
    \]
\end{definition}

This condition on the initial data (for any $t>0$ fixed) is both necessary and sufficient to guarantee that the KPZ fixed point (at time $t>0$) does not explode, see \cite[Proposition 6.1]{sarkar2021brownian}.

Starting from admissible data $h_0\in \mathscr{I}_t$, $t>0$ the KPZ fixed point at time $t> 0$, $\mathfrak{h}_t(\cdot, h_0)$ (or $\mathfrak{h}(\cdot)$ when $t, h_0$ are clear from the context), can be expressed in terms of a variational formula involving the directed landscape $\mathcal{L}$:
\begin{equation}\label{eq: variational formula}
    \mathfrak{h}_t(y, h_0) = \max_{x\in \R}(h_0(x) + \mathcal{L}(x, 0; y,t))\,,\quad y\in \R\,.
\end{equation}

We denote the law of
$\mathfrak{h}_t(y, h_0)-\mathfrak{h}_t(a, h_0)\,, y\in [a, b]$
for $a< b$, supported on $\mathscr{C}_0([a, b]; \R)$, by $\nu^{h_0}$ (suppressing dependence on $a, b$ as it will always be clear from the context). 

Upper semicontinuous functions are nicely compatible with the variational formula of the KPZ fixed point. One can always replace the full variational formula \eqref{eq: variational formula} with a one over a fixed countable dense subset of the `max-plus' support of the initial data.

\section{Brownian mutual absolute continuity of the KPZ fixed point}\label{sec: mut abs cont BM}

In this section, we prove for initial data in $\mathscr{I}_t$ for some $t> 0$ (see Definition \ref{def: finitary}), one obtains mutual absolute continuity of the laws of the spatial increments of the KPZ fixed point against rate two Brownian motion on compacts. We crucially leverage the `full-space' coupling between the Airy sheet and Airy line ensemble on $\R^2$, \eqref{eq: airy sheet coupling full space}. Figure \ref{fig: flowchart mut abs cont} shows the key steps of the proof of the main mutual absolute continuity result, Theorem~ \ref{thm: mut abs cont finitary init data}.

We begin with a proposition regarding the monotonicity of some functionals of the Airy line ensemble and compactly supported initial data (appearing as boundary data in the variational characterisation of the KPZ fixed point; see the discussion in \cite[Section 5]{tassopoulos2025quantitativebrownianregularitykpz}).

First, recall from \cite[Theorem 3.7]{sarkar2021brownian} the notation the semi-infinite last passage values for $x> 0$, $\ell \ge 1$
\begin{equation}\label{eq: semi-inf lpp airy}
    \mathcal{A}[x\to (0, \ell)]\stackrel{\mathrm{def}}{=} \lim_{k\to \infty}(\mathcal{A}[x_k \to (0, \ell)]-\mathcal{A}[x_k \to (0, 1)])+\mathcal{S}(x, 0)\,,
\end{equation}
for $x_k = (-\sqrt{k/2x}, k)$, $k\ge 1$. By \cite[Lemma 3.8]{sarkar2021brownian}, they are non-decreasing in $\ell$ (hence finite) and by \cite[Lemma 3.9]{sarkar2021brownian} $\mathscr{F}_{-}\stackrel{\mathrm{def}}{=}\sigma(\{\mathcal{A}_i(x):x\leq 0,i=1,2,\cdots\})$-measurable. 

\begin{proposition}\label{prop: separation of initial data}
    Fix finitary initial data $h_0\in \mathrm{UC}_c$ with `max-plus' support $\supp{h_0}\subseteq (0,\infty)$ and denote the random `boundary data' as 
    \[
    G_\ell = \max_{x\in \R}(h_0(x) + \mathcal{A}[x\to (0, \ell)])\,.
    \]
    Then, almost surely, $G_m > G_{m+1}$ for $m\ge 1$. Moreover, if the initial data is compactly-supported such that $\supp{h_0}\subseteq (0,\beta)$ for some $\beta>0$, we obtain the almost sure uniform lower bounds for $m > m'$,
    \begin{align*}
    G_m-G_{m'} \ge \mathcal{A}[(\varepsilon^\infty_{\beta, m'}, m')\to (0, m)]-\mathcal{A}[(\varepsilon^\infty_{\beta, m'}, m')\to (0,m')] >0\,,
    \end{align*}
    where $\varepsilon^\infty_{\beta, m'}$ is the first time the semi-infinite geodesic $\pi[\beta \to (0, m')]$ reaches level $m'$.
\end{proposition}

\begin{proof}
    We show the strict monotonicity for $m=1$. The other cases (and corresponding lower bounds) follow analogously. We can express almost surely from \cite[Proposition 5.1]{sarkar2021brownian},
    $G_1 = \max_{x\in \R}(h_0(x)+\mathcal{S}(x,0))$,
    and
    \[
    G_2 = \max_{x\in \R}(h_0(x)+\lim_{k\to \infty}(\mathcal{A}[x_k\to(0, 2)]-\mathcal{A}[x_k\to(0, 1)]+\mathcal{S}(x,0))\,,
    \]
    where $x_k = (-\sqrt{k/(2x)}, k)$, $x > 0$, $k\ge 1$.  

    By monotonicity of geodesics, we have  for any fixed $0<x<y$, $k\ge 1$ almost surely, $ \varepsilon^k_{x,2} \le \varepsilon^k_{y, 2}$, where for $x>0$, $2\le \ell\le k$, $\varepsilon^k_{x,\ell}$ is the first time the geodesic from $x_k$ to $(0,2)$ reaches level $\ell$. Note for any fixed $x > 0$, as $k\to \infty$, $\varepsilon^k_{x,\ell}$ eventually stabilises to the respective jump time of the semi-infinite geodesic on the Airy line ensemble.
    
    Now, by the metric composition law for last passage percolation, we have almost surely eventually in $k\ge 1$, for any $x\in \supp{h_0}\subseteq (0, \beta)$ (where we assume the support of $h_0$ is countable and $\beta = \sup \supp{h_0}<\infty$),
    \begin{align*}
    \mathcal{A}[x_k\to(0, 2)]-\mathcal{A}[x_k\to(0, 1)]&\le \mathcal{A}[x_k\to(0, 2)]-\mathcal{A}[x_k\to(\varepsilon^k_{\beta, 2}, 2)]-\mathcal{A}[(\varepsilon^k_{\beta, 2}, 2)\to (0, 1)]\\
    &=\mathcal{A}[x_k\to(\varepsilon^k_{\beta, 2}, 2)]+\mathcal{A}[(\varepsilon^k_{\beta, 2}, 2)\to (0,2)]\qquad \hfill (\varepsilon^k_{x, 2}\le \varepsilon^k_{\beta, 2})\\
    &-\mathcal{A}[x_k\to(\varepsilon^k_{\beta, 2}, 2)]-\mathcal{A}[(\varepsilon^k_{\beta, 2}, 2)\to (0, 1)]\,,
    \end{align*}
    where $\varepsilon^k_{\beta, 2}< 0$ almost surely, which follows from from the locally Brownian nature of the Airy line ensemble, see Lemma \ref{lemma: semi-inf geod jump time pos} in the Appendix.
    
    We thus have
    \begin{align*}
    \mathcal{A}[x_k\to(0, 2)]-\mathcal{A}[x_k\to(0, 1)]&\le\mathcal{A}[(\varepsilon^k_{\beta, 2}, 2)\to (0,2)]-\mathcal{A}[(\varepsilon^k_{\beta, 2}, 2)\to (0, 1)]\,.
    \end{align*}
    Taking $k\to \infty$ we obtain the almost sure bounds
    \begin{align*}
    G_2 &= \max_{x\in \R}(h_0(x)+\lim_{k\to \infty}(\mathcal{A}[x_k\to(0, 2)]-\mathcal{A}[x_k\to(0, 1)]+\mathcal{S}(x,0))\\
    &\le G_1 + \lim_{k\to \infty}(\mathcal{A}[(\varepsilon^k_{\beta, 2}, 2)\to (0,2)]-\mathcal{A}[(\varepsilon^k_{\beta, 2}, 2)\to (0, 1)])\\
    &= G_1 + \mathcal{A}[(\varepsilon^\infty_{\beta, 2}, 2)\to (0,2)]-\mathcal{A}[(\varepsilon^\infty_{\beta, 2}, 2)\to (0, 1)]< G_1\,,
    \end{align*}
    where the last strict inequality follows from the (mutual) absolute continuity of the centred Airy line ensemble with Brownian motion on compacts (cf. \cite[Theorem 1.1]{dauvergne2024wienerdensitiesairyline}, and the fact that the jump times $\varepsilon^k_{\beta, 2}$ converge as $k\to \infty$ to some $\varepsilon^\infty_{\beta, 2}< 0$ (recall from Section \ref{sec: airy sheet} and Lemma \ref{lemma: semi-inf geod jump time pos} again with $x = \varepsilon^\infty_{\beta, 2}, k = 1, \ell = 2$), which is in fact the first time the semi-infinite geodesic $\pi[\beta \to (0, 2)]$ reaches level $2$. 
\end{proof}

Using the positive gap between the first two values of the boundary data and continuity of the infinite Skorokhod reflections of lines in the Airy line ensemble (appropriately shifted by the boundary data, cf. Lemma \ref{lemma: infinite lpp cont at zero}), we obtain the following `local' result for compactly supported initial data (recall Definition \ref{def: comp supp UC}). More precisely, `locally', the increments of the KPZ fixed point look up to a random time horizon like those of the parabolic Airy$_2$ process.

\begin{proposition}\label{prop: local comparison KPZ random time}
    Fix initial data $h_0\in \mathscr{I}_t$ with $\inf \supp{h_0} > -\infty$. Then, for every $y\in \R$, there exists some random open interval $J_y\subseteq [y, \infty)$ (in the subspace topology of $[y, \infty)$) such that,
        \[
        \mathfrak{h}_t(\cdot, h_0)|_{J_y} - \mathfrak{h}_t(y, h_0) \stackrel{\mathrm{law}}{=} t^{\frac{1}{3}}(A(t^{-\frac{2}{3}}\;\cdot)|_{J'_y} - A(t^{-\frac{2}{3}}y)\,,
        \]
        where $A(\cdot)$ is the parabolic $\mathrm{Airy}_2$ process and $y\in J'_y\subseteq [y, \infty)$ is another random open interval depending on the law of the Airy line ensemble, which contains $y$ in its interior.
\end{proposition}

\begin{proof}
        By $1:2:3$ scaling and translation symmetries of the Airy sheet, it suffices to take $t=1$, $\inf\supp{h_0} > 0$ and $y > 0$. We can now represent using the coupling \eqref{eq: airy sheet coupling full space} (cf. \cite[Proposition 5.1]{sarkar2021brownian})
        \[
        \mathfrak{h}(s) = \mathfrak{h}_1(s, h_0) = \max_{\ell \ge 1}(G_\ell + \mathcal{A}[(y, \ell)\to (s, 1)])\,,\qquad s\ge y\,,
        \]
        with
        \[
        G_\ell \equiv \max_{x\in \R}\bigg(h_0(x)+\lim_{k\to \infty}(\mathcal{A}[x_k\to(y, \ell)]-\mathcal{A}[x_k\to(y, 1)])+\mathcal{S}(x,y)\bigg)\,,
        \]
        where $x_k = (-\sqrt{k/(2x)}, k)$, $x > 0$, $k\ge 1$.
        
        Now, by the shape estimates of the Airy sheet, \eqref{eq: airy shape bnds} (cf. \cite[Proposition 6.1]{sarkar2021brownian}), there exists a random $m^*\in \N$ such that almost surely, we can represent the KPZ fixed point $\mathfrak{h}$ on $[y, y+1]$ as
        \[
         \mathfrak{h}(s) = \max_{\ell \ge 1}(G^{m^*}_\ell + \mathcal{A}[(y, \ell)\to (s, 1)])\,,\qquad s\in [y, y+1]\,,
        \]
        where
        \[
        G^{m^*}_\ell = \max_{x\in [-m^*, m^*]}(h_0(x) + \mathcal{A}[x\to (y, \ell)])\,,\quad \ell \ge 1\,.
        \]
        By \cite[Proposition 5.1]{sarkar2021brownian}, there exists a random $L_0$ such that
        we can represent the KPZ fixed point $\mathfrak{h}$ on $[y, y+1]$ as
        \begin{align*}
             \mathfrak{h}(s) &= \max_{1\le \ell \le L_0}(G^{m^*}_\ell + \mathcal{A}[(y, \ell)\to (s, 1)])= G^{m^*}_1 + \mathcal{A}_1(s)-\mathcal{A}_1(y) \\
             &+ \max_{y\le r\le s}\bigg(\max_{2\le \ell \le L_0}(G^{m^*}_\ell + \mathcal{A}[(y, \ell)\to (r, 2)])-G^{m^*}_1 - \mathcal{A}_1(r)+\mathcal{A}_1(y)\bigg)\lor 0\,,\quad s\in [y, y+1]\,.
        \end{align*}
        By Proposition \ref{prop: separation of initial data}, we have $G^{m^*}_1 > G^{m^*}_2$ almost surely (by sectioning on events $\{m^* = m \}$, $m \ge 1$). Thus, since $L_0, m^*$ are uniform for $s\in [y, y+1]$, by continuity of last passage values and the monotonicity of the boundary data $G^{m^*}_\ell$ (cf. Proposition \ref{prop: separation of initial data}), we have
        \[
        \lim_{s\to y}\max_{2\le \ell \le L_0}(G^{m^*}_\ell + \mathcal{A}[(y, \ell)\to (s, 2)]) = G^{m^*}_2 < G^{m^*}_1\,.
        \]
        This means there exists a random neighbourhood $J_y$ such that
        \[
        \mathfrak{h}(\cdot)|_{J_y} - \mathfrak{h}(y) = \mathcal{A}(\cdot)|_{J_y} - \mathcal{A}(y)\,.
        \]
        Then, take $J'_y \stackrel{\mathrm{law}}{=} J_y \text{ conditioned on } \mathcal{A}$ to conclude the proof.
\end{proof}

In particular, inspecting the proof of Proposition \ref{prop: local comparison KPZ random time}, we obtain in the following proposition a coalescence result for the KPZ fixed point started from different initial data, see Figure \ref{fig: KPZ coal}. 

\begin{proposition}\label{prop: KPZ local coal inc}
    Fix $t> 0$ and $\alpha, y\in \R$. Then, there exists a coupling of
    \[
    (\mathfrak{h}_t(\cdot, h_0)\,: h_0 \in \mathscr{I}_t\,, \inf\supp{h_0}> -\alpha)
    \]
    such that almost surely, for any two $h_0, h'_0$ as above, there exists some random $\tau_y = \tau_y(h_0, h'_0) > y$ with
    \[
    \mathfrak{h}_t(\cdot, h_0)|_{[y, \tau_y]} - \mathfrak{h}_t(y, h_0) =  \mathfrak{h}_t(\cdot, h'_0)|_{[y, \tau_y]} - \mathfrak{h}_t(y, h'_0)\,,
    \]
    that is the increments of the KPZ fixed points eventually coalesce.
\end{proposition}

\begin{figure}[ht]
    \centering
    \includegraphics[width=0.6\textwidth]{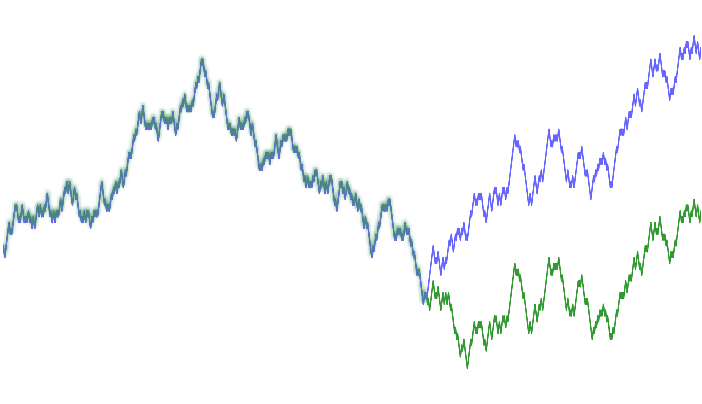}
    \caption{Illustration of the coupling between the KPZ fixed points on $[0,2]$ started from compactly supported flat $0$ initial data on $[0, 1]$, that is $0\cdot \delta_{[0, 1]}$ (recall \eqref{eq: max-plus indicator}) (\color{ForestGreen} green\color{black}) and the superposition of two narrow wedges at $0, 1$, that is $0\cdot \delta_{\{0, 1\}}$ (\color{blue} blue\color{black}).}
    \label{fig: KPZ coal}
\end{figure}

We now prove a continuity result for the infinite last passage representation of the KPZ fixed point at the origin. This is the content of the following lemma.

\begin{lemma}\label{lemma: infinite lpp cont at zero}
 Fix finitary initial data with bounded `max-plus' support $\supp{h_0}\subseteq (0,\infty)$ and denote the random `boundary data' as $G_\ell = \max_{x\in \R}(h_0(x) + \mathcal{A}[x\to (0, \ell)])$, $\ell \ge 1$ (cf. \eqref{eq: semi-inf lpp airy}). Then, for any $k\ge 1$ and $t\in K$, $K\subseteq [0, \infty)$ compact, there exists a random almost surely finite $L^k_0$ such that
    \[
    \max_{\substack{k\le \ell}} (G_\ell + \mathcal{A}[(0,\ell)\to (s, k)]) = \max_{\substack{k\le \ell\le L^k_0 }} (G_\ell + \mathcal{A}[(0,\ell)\to (s, k)])\,,\quad \text{ for all } 0\le s\le t\,.
    \]
    In particular, we have the continuity near the origin
     \[
        \lim_{t\searrow 0}\max_{\substack{k\le \ell}} (G_\ell + \mathcal{A}[(0,\ell)\to (t, k)]) = G_k\,,\qquad k\ge 1\,.
        \]
\end{lemma}

\begin{proof}
    We first show that for any $k\ge 1$ and $t\in K$, $K\subseteq [0, \infty)$ compact, there exists a random almost surely finite $L^k_0$ such that
    \[
    \max_{\substack{k\le \ell}} (G_\ell + \mathcal{A}[(0,\ell)\to (s, k)]) = \max_{\substack{k\le \ell\le L^k_0 }} (G_\ell + \mathcal{A}[(0,\ell)\to (s, k)])\,,\quad \text{ for all } 0\le s\le t\,.
    \]
    The rest of the proof would follow by the continuity of last passage values (over the continuous environment).

    Indeed, fix $k\ge 1, t\ge 0, K\subseteq \R$ as above. Also set $x_0 = \sup \supp{h_0} >0$ and for $x>0$, $n\ge 1$, $x^n_0 = (-\sqrt{n/2x}, n)$. Now, by \cite[Lemma 3.6]{sarkar2021brownian}, there exists a random $Y\in \N$ such that the (rightmost) semi-infinite geodesic starting from $(Y, 1)$, $\pi[x, Y]$ satisfies $\pi[x, Y](\sup K)\ge k$. Moreover, for any fixed $\ell\ge k$, there exists a random $N_{x, \ell}\in \N$, such that the rightmost geodesics $\pi[x, Y],\pi[x^n_0, (Y,1)],\pi[x, (0, 1)], \pi[x^n_0, (0,1)]$, pass through a common random point $(T, d(T))$ with $T\le 0 $ for all $n\ge N_{x, \ell}$. Now, by the ordering of geodesics, \cite[Proposition 2.8]{sarkar2021brownian}, we have for any fixed $x\in \supp{h_0}$, almost surely, $\pi[x^n, 1]\le \pi[x^n, (\sup K,\ell)]\le \pi[x^n, Y]\le \pi[x, Y]$ for all $n\ge N_{x, \ell}$.
    Hence, we have for all $t\in K$ by the ordering of geodesics, almost surely, uniformly in $t\in K$
    \[
    \mathcal{A}[x^n \to (t, k)] = \max_{k\le \ell\le L^k_0}\mathcal{A}[x^n \to (0, \ell)]+ \mathcal{A}[(0, \ell)\to (t, k)]\,,\qquad n\ge N_{x, \ell}\,,
    \]
    with $L^k_0 = \pi[x, Y](0) \le \pi[x_0, Y](0)$ (by \cite[Lemma 3.5]{sarkar2021brownian}). Thus, we have by \cite[Theorem~ 3.7]{sarkar2021brownian}, for $\ell \ge k$, $x\in \supp{h_0}$,
    \begin{align*}
    \max_{k\le \ell} (\mathcal{A}[x \to (0, \ell)]+ \mathcal{A}[(0,\ell)\to (t, k)])&= \lim_{n\to \infty}\mathcal{A}[x^n \to (t, k)]-\mathcal{A}[x^n \to (t, 1)]\\
    &= \max_{k\le \ell\le L^k_0}\mathcal{A}[x \to (0, \ell)]+ \mathcal{A}[(0, \ell)\to (t, k)]\,.
    \end{align*}
    Since the $L^k_0$ is uniform in $x$, and the support of the initial data can always be taken to be countable, we conclude almost surely, for all $t\in K$,
    \[
    \max_{k\le \ell} (G_\ell + \mathcal{A}[(0,\ell)\to (t, k)]) = \max_{k\le \ell\le L^k_0} (G_\ell + \mathcal{A}[(0,\ell)\to (t, k)])\,,
    \]
    which concludes the proof.
\end{proof}

\begin{remark}
    The above proof shows that one can make sense of the `infinite' Skorokhod reflections
    \[
       \max_{\ell\ge m}(G_{\ell}+\mathcal{A}[(0,\ell)\to (s,m)])\,, \quad s\ge 0\,, m\ge 1
       \]
    as random continuous functions on the entire path space $\mathscr{C}(\R;\R)$, since on compacts by geodesic geometry, there is an almost surely finite random maximiser. This means we can represent the KPZ fixed point  at unit time started from $h_0$,
     \[
       \mathfrak{h}_1(s, h_0) = \max_{\ell\ge 1}(G_{\ell}+\mathcal{A}[(0,\ell)\to (s,1)])\,, \quad s\ge 0\,, m\ge 1\,.
       \]
    For more details, see \cite[Proposition 5.1]{sarkar2021brownian} and \cite[Section 5]{tassopoulos2025quantitativebrownianregularitykpz}.
\end{remark}

We are now in a position to prove that the increments of the KPZ fixed point started from compact initial data are mutually absolutely continuous with respect to Brownian motion (see Figure \ref{fig: gibbs resampling} for an illustration).

  \begin{figure}[ht]
   \includegraphics[width = 0.7\textwidth]{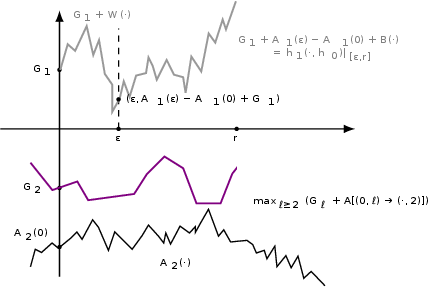}
       \caption{Illustration of the Brownian Gibbs resampling for the top line of the Airy line ensemble appearing in the variational expression for the KPZ fixed point on the compact interval $[0, r]$. In short, it can be expressed (up to mutual absolute continuity) as a concatenation of a Brownian bridge $W$ and Brownian motion $B$ (conditionally independent given the Airy line ensemble) at the point $(\varepsilon, \mathcal{A}_1(\varepsilon-\mathcal{A}_1(0)+G_1)$, conditioned to not hit $\mathcal{A}_2$. The increments of the KPZ fixed point in the interior interval $[\varepsilon, r]$ on the event $\mathcal{A}_1(\cdot)-\mathcal{A}_1(0)+G_1$ avoids $\max_{\ell \ge 2}(G_\ell + \mathcal{A}[(0, \ell)\to (\cdot,2)])$ are simply the increments of the Brownian motion $B$. By standard monotonicity results for Brownian bridges, conditionally on $B|_{[\varepsilon, r]}$, the above non-intersection condition can always be ensured to occur with positive probability. }
       \label{fig: gibbs resampling}
   \end{figure}

\begin{theorem}\label{thm: mutual abs cont}
     Fix $t>0$ and let $h_0 \in \mathrm{UC}_c$, that is, its max-plus support (cf. Definition \ref{def: max support}) is bounded. Then, for any $y,\in \R$ and $r>0$, we have the (mutual) absolute continuity relation 
    $\nu^{h_0} \ll \mu\ll \nu^{h_0}$, 
    on paths on $\mathscr{C}([y, y+r];\R)$
    where $\nu^{h_0}$ denotes the law of the increments of the KPZ fixed point started from $h_0$ and $\mu$ the appropriate restriction of the rate two Wiener measure.
\end{theorem}
\begin{remark}
    The above mutual absolute continuity result can be partially extended to the full Airy sheet, see Theorem \ref{thm: abs cont airy sheet}.
\end{remark}

\begin{proof}By the local Brownianness of the KPZ fixed point, (cf. \cite{sarkar2021brownian}) it suffices to show $ \mu\ll \nu^{h_0}$.

     By $1:2:3$-scaling, we can without loss of generality set $t = 1$. By the translation symmetries of the Airy sheet, we can also set $y=0$, and fix the support of the initial data to lie in $[0, \infty)$.
    
    Now, suppose $A\subseteq \mathscr{C}([\varepsilon, r]; \R)$, $0\le \varepsilon < r$ only depending on increments, such that $\nu^{h_0}(A) = 0$. Then we estimate (by the metric composition law for last passage values),
    \begin{align*}
        0 &=\nu^{h_0}(A)= \PP(\mathfrak{h}_1(\cdot, h_0)\in A)\\
        &\ge \PP(\mathcal{A}_1(\cdot)\in A\,,\max_{\ell \ge 2}(G_\ell + \mathcal{A}[(0, \ell)\to (s, 2)])\le G_1 + \mathcal{A}_1(s) - \mathcal{A}_1(0)\,, s\in [0, r])\,.
    \end{align*}
    Conditioning on the sigma algebra $\mathscr{F} \equiv \sigma(\{\mathcal{A}_i(x): x\not \in (\varepsilon, r+1)\times \{1\}\}$,
    and using the Brownian Gibbs property, conditionally on $\mathscr{F}$ the absolute continuity relation $\tilde{\mu} \ll \nu$ on $\mathscr{C}([\varepsilon, r],\R)$
    where $\tilde{\mu}$ is the law of a rate two Brownian motion $B$ conditioned on the event $\mathcal{A}_1(\varepsilon) + B(s) > \mathcal{A}_{2}(s)$, for all $s\in [\varepsilon, r]$ ($B$ has the law of the rate two Wiener measure $\mu$) and $\nu$ is the conditional law of the increments of $\mathcal{A}_1(y)-\mathcal{A}_1(\varepsilon)\,, y\in [\varepsilon, r]$.
    Hence, we have
    \begin{align*}
        0 &= \PP\bigg(B(\cdot)\in A\,,\max_{\ell \ge 2}(G_\ell + \mathcal{A}[(0, \ell)\to (s, 2)])\le G_1 + \mathcal{A}_1(s) - \mathcal{A}_1(0)\,, s\in [0, \varepsilon]\,,\\
        &\max_{\ell \ge 2}(G_\ell + \mathcal{A}[(0, \ell)\to (s, 2)])\le G_1 + B(s)+ \mathcal{A}_1(\varepsilon) - \mathcal{A}_1(0)\,, s\in [\varepsilon, r]\,,\\
        &\mathcal{A}_1(\varepsilon) + B(s) > \mathcal{A}_{2}(s)\,, s\in [\varepsilon, r]\bigg)\,.
    \end{align*}
    Now, conditioning on the sigma algebra 
   $\mathscr{F}' \equiv \sigma(\{\mathcal{A}_i(x): x\not \in (0, \varepsilon)\times \{1\}\}$,
    and using the Brownian Gibbs property, we have as before
    \begin{align*}
        0 &= \PP\bigg(B(\cdot)\in A\,,\max_{\ell \ge 2}(G_\ell + \mathcal{A}[(0, \ell)\to (s, 2)])\le G_1 + W(s)\,, s\in [0, \varepsilon]\,,\\
        &\mathcal{A}_1(0) + W(s) > \mathcal{A}_{2}(s)\,, s\in [0, \varepsilon]\,,\mathcal{A}_1(\varepsilon) + B(s) > \mathcal{A}_{2}(s)\,, s\in [\varepsilon, r]\,,\\
        &\max_{\ell \ge 2}(G_\ell + \mathcal{A}[(0, \ell)\to (s, 2)])\le G_1 + B(s)+ \mathcal{A}_1(\varepsilon) - \mathcal{A}_1(0)\,, s\in [\varepsilon, r]\bigg)\\
        &= \PP(B(\cdot)\in A\,, E)\,,
    \end{align*}
    where 
    \begin{align*}
        E&\equiv \bigg\{\max_{\ell \ge 2}(G_\ell + \mathcal{A}[(0, \ell)\to (s, 2)])\le G_1 + W(s)\,, s\in [0, \varepsilon]\,,\\
        &\mathcal{A}_1(0) + W(s) > \mathcal{A}_{2}(s)\,, s\in [0, \varepsilon]\,,\mathcal{A}_1(\varepsilon) + B(s) > \mathcal{A}_{2}(s)\,, s\in [\varepsilon, r]\,,\\
        &\max_{\ell \ge 2}(G_\ell + \mathcal{A}[(0, \ell)\to (s, 2)])\le G_1 + B(s)+ \mathcal{A}_1(\varepsilon) - \mathcal{A}_1(0)\,, s\in [\varepsilon, r]\bigg\}\,,
    \end{align*}
    $W$ is a rate two Brownian bridge starting at $(0,0)$ and ending at $(\varepsilon, \mathcal{A}_1(\varepsilon)-\mathcal{A}_1(0))$ (independent from $B$, see Figure \ref{fig: gibbs resampling}). Now, by the Brownian Gibbs property and the fact that conditionally on
    $\mathscr{F}'' \equiv \sigma(\{\mathcal{A}_i(x): x\not \in (\varepsilon/2, r)\times \{1\}\})$,
    the stochastic domination holds (cf. Lemma \ref{lemma: bridge monotonicity}):
    $\mathcal{A}_1(\varepsilon) \ge_d \tilde{W}_1(\varepsilon)$,
    where for two random variables $X,Y$, $X\le_d Y$ denotes the stochastic domination of $X$ by $Y$ and $\tilde{W}_1(\cdot)$ has the law of a rate two Brownian bridge starting from $(\varepsilon/2, \mathcal{A}_1(\varepsilon/2))$ and ending at $(r, \mathcal{A}_1(r))$.  Moreover, we can express $W$ as $W_0 + L$, where $W_0$ is a Brownian bridge vanishing at both endpoints and $L$ an affine function with endpoints $(0, 0)$ and $(\varepsilon, \mathcal{A}_1(\varepsilon)-\mathcal{A}_1(0))$. Now, conditionally on $\sigma(B|_{[\varepsilon , r]})\lor \mathscr{F}''$, we have by Lemma \ref{lemma: support bb} and the above for any $a \in \R, \eta > 0$,
    \begin{align*}
        \PP(\mathcal{A}_1(\varepsilon) > a\,, \norm{W_0}_{L^\infty}\le \eta | \sigma(B|_{[\varepsilon , r]})\lor \mathscr{F}'')&= \PP(\mathcal{A}_1(\varepsilon) > a | \sigma(B|_{[\varepsilon , r]})\lor \mathscr{F}'')\cdot \PP(\norm{W_0}_{L^\infty}\le \eta)\\
        &\ge  \PP(\tilde{W}_1(\varepsilon) > a| \sigma(B|_{[\varepsilon , r])})\cdot \PP(\norm{W_0}_{L^\infty}\le \eta)> 0
    \end{align*}
    almost surely. Moreover, conditioning on $\sigma(B|_{[\varepsilon , r]})\lor \mathscr{F}''$, we can treat the data
    \begin{equation}\label{eq: data}
        \bigg(B, \max_{\ell \ge 2}(G_\ell + \mathcal{A}[(0, \ell)\to (s, 2)]), \mathcal{A}_2, \mathcal{A}_1(0), G_1\bigg)\in \mathscr{C}([\varepsilon, r]; \R)\times \mathscr{C}^2([0, r]; \R)\times \R^2
    \end{equation}
    as fixed. Taking $\mathcal{A}_1(\varepsilon)$ sufficiently large and $\varepsilon >0$ sufficiently small, which will depend on the data \eqref{eq: data}, we obtain $\PP(E\, |\sigma(B|_{[\varepsilon , r]})\lor \mathscr{F}'' )>0$ almost surely.
    
   Since $0=\PP(B(\cdot)\in A\,, E)=\mathbb{E}[\mathbf{1}(B(\cdot)\in A) \cdot \PP(E\, |\sigma(B|_{[\varepsilon , r]})\lor \mathscr{F}'' )]$, and $\PP(E\, |\sigma(B|_{[\varepsilon , r]})\lor \mathscr{F}'' )>0$,
    we finally obtain $\PP(B(\cdot)\in A) = 0$, concluding the proof.
   \end{proof}

   \begin{remark}
       Observe that the same argument in conjunction with Proposition \ref{prop: separation of initial data} gives for any $m\ge 1$, the laws of the increments of the `infinite' Skorokhod reflections
       \[
       \max_{\ell\ge m}(G_{\ell}+\mathcal{A}[(0,\ell)\to (s,m)])\,, \quad s\ge 0
       \]
       are mutually absolutely continuous with respect to the law of a rate two Brownian motion on paths on $\mathscr{C}([\varepsilon, t]; \R)$, for any $0 < \varepsilon < t$ (even conditionally on $\mathscr{F}_-$). 
   \end{remark}

We now prove that if the support of the initial data is bounded in at least one direction, one can extend the above mutual absolute continuity on compacts. 

\begin{theorem}\label{thm: mut abs cont unbounded supp init data}
    Let $h_0\in \mathscr{I}_t$ be such that $\supp{h_0}$ is bounded in at least one direction. Then, for any compact $K\subset \R$, $t>0$ we have the (mutual) absolute continuity relation 
    $\mu \ll \nu^{h_0}\ll \mu$, on paths on $\mathscr{C}_0(K;\R)$
    where $\nu^{h_0}$ denotes the law of the increments of the KPZ fixed point started from $h_0$ at time $t>0$ in $K$.
\end{theorem}
       \begin{proof}
       By the local Brownianness of the KPZ fixed point, (cf. \cite{sarkar2021brownian}) it suffices to show $ \mu\ll \nu^{h_0}$.

       By $1:2:3$-scaling, we can without loss of generality set $t = 1$. Also, the translation and reflection symmetries of the Airy sheet allow us to set $y>0$ (by enlarging $K$ if necessary), and by taking the supports to be bounded from the left, fix the support of the initial data to lie in $(0, \infty)$ (since it is compact).

       Recall the notation for the sigma algebra $\mathscr{F}_- = \sigma(\{\mathcal{A}_i(x): i\ge 1\,, x\le 0\})$. Now, by the coupling \ref{eq: airy sheet coupling full space} and the metric composition law \eqref{eq: composition}, we can express the KPZ fixed point as (recall the notation for the Pitman transform \eqref{eq: pitmantrans}), 
       \[
       \mathfrak{h}(s) = \max_{\ell\ge 1}(G^\infty_\ell + \mathcal{A}[(0, \ell) \to (s, 1)]) = WF_1(s)\,,\quad s\in [y, y+r]
       \]
    with 
        \[
        F = \bigg(G^\infty_1 + \mathcal{A}_1(\cdot) - \mathcal{A}_1(0)\,, \max_{2\le \ell}(G^\infty_\ell + \mathcal{A}[(0, \ell)\to(\cdot,2)])\bigg)\,,
        \]
        where
    \[
    G^\infty_\ell \equiv \max_{x\in \R}(h_0(x)+\mathcal{A}[x\to (0, \ell)])
    \]
    and 
    \[
    G^m_\ell \equiv \max_{x\in (m-1,m]}(h_0(x)+\mathcal{A}[x\to (0, \ell)])\,, \quad m\ge 1\,.
    \]
    Moreover, we have
       \[
       \max_{2\le \ell}(G^\infty_\ell + \mathcal{A}[(0, \ell)\to(s,2)]) = \max_{m\ge 1}\max_{2\le \ell \le L^m_0 }(G^{m}_\ell + \mathcal{A}[(0, \ell)\to(s,2)]) \,,\quad s \in [y, y+r]
       \]
       where $L_0^m$ is the semi-infinite geodesic intercept $\pi[m\to y+r](0)$ (cf. \eqref{def: semi-inf geo}). We now show this maximum is attained for an almost surely finite $m^*\ge 1$ uniformly in $s\in [y, y+r]$. Moreover, $m^*$ is measurable with respect to $\mathscr{F}_-$. 

       Indeed, we have as $m\to \infty$, coupling the Airy line ensemble to an Airy sheet $\mathcal{S}(\cdot, \cdot)$ using \eqref{eq: airy sheet coupling full space}, {\begin{align*}
            \max_{2\le \ell \le L^m_0 }(G^{m}_\ell + \mathcal{A}[(0, \ell)\to(s,2)])&\le \max_{1\le \ell \le L^m_0 }(G^{m}_\ell + \mathcal{A}[(0, \ell)\to(s,1)])\\
            &\le \max_{x\in (m -1, m]}(h_0(x)-x^2) + \max_{s\in [0, y+r]}\max_{x\in (m-1, m ]}(\mathcal{S}(x, s)+x^2)\\
            &\le (m-1)\cdot\max_{x\in (m -1, m]}\frac{h_0(x)-x^2}{|x|}+\mathfrak{C}+2m(y+r)\\
            & +c\log^{\frac{2}{3}}(2+ |m|+|y+r|)\,,\quad s\in [0, y+r]
            \end{align*}}
        for some absolute $c> 0$ and almost surely finite $\mathscr{G}$-measurable $\mathfrak{C}$ where the last bound follows from the Airy sheet shape estimates \eqref{eq: airy shape bnds}. Since $h_0$ is finitary (see Definition \ref{def: finitary}), we have
        \[
        \lim_{m\to \infty} \max_{x\in (m -1, m]}\frac{h_0(x)-x^2}{|x|} = -\infty\,,
        \]
        and so for all $m\ge 1$ sufficiently large,
        \begin{align*}
            \max_{2\le \ell \le L^m_0 }(G^{m}_\ell + \mathcal{A}[(0, \ell)\to(s,2)])&\le \max_{2\le \ell \le L^1_0 }(G^{1}_\ell + \mathcal{A}[(0, \ell)\to(s,2)])\,,\quad s\in [0, y+r]\,.
       \end{align*}
       Hence, we have almost surely,
       \begin{equation}\label{eq: inf supp ef trunc}
       \max_{\substack{2\le \ell \le L^m_0 \\m\ge 1}}(G^{m}_\ell + \mathcal{A}[(0, \ell)\to(s,2)]) = \max_{\substack{2\le \ell \le L^m_0 \\1\le m\le m^*}}(G^{m}_\ell + \mathcal{A}[(0, \ell)\to(s,2)])\in \mathscr{C}([0, y+r];\R)\,,
       \end{equation}
       where the continuity follows from Lemma \ref{lemma: infinite lpp cont at zero}. We now conclude 
       \[
       \lim_{s\to 0}\max_{2\le \ell}(G^\infty_\ell + \mathcal{A}[(0, \ell)\to(s,2)]) = G_2^\infty < G^\infty_1\,,
       \]
       the second almost sure strict inequality also follows from the above argument, since $h_0(x) \to -\infty$ as $x\to \infty$ and Proposition \ref{prop: separation of initial data}. 

       To complete the argument, one now proceeds exactly as in the proof of Theorem~ \ref{thm: mutual abs cont}.

       The case of an unbounded `max-plus' support to the left is entirely analogous and follows by the symmetries of the Airy sheet, \eqref{eq: Airy sheet symmetry} and the time-reversal symmetry of Brownian motion.
    \end{proof}

Finally, we now prove that for arbitrary initial data with unbounded support, the law of the increments of the KPZ fixed point is mutually absolutely continuous against the Wiener measure. This crucially uses the fact that the Airy sheet can be constructed as a deterministic function of the Airy line ensemble on the entire plane.

\begin{theorem}\label{thm: mut abs cont finitary init data}
    Let $t>0$ and $h_0\in \mathscr{I}_t$. Then, for any compact $K\subset \R$, have the (mutual) absolute continuity relation
    $\mu \ll \nu^{h_0}\ll \mu$ on paths on $\mathscr{C}_0(K;\R)$,
    where $\nu^{h_0}$ denotes the law of the increments of the KPZ fixed point started from $h_0$ at time $t>0$ in $K$.
\end{theorem}

\begin{proof}

From \cite[Theorem~ 1.2]{sarkar2021brownian}, we have $\nu^{h_0}\ll \mu$. It thus remains to prove that $\mu \ll \nu^{h_0}$. 

As before, we lose no generality in assuming that $t=1$ and $h_0 \in \mathscr{C}(\R; \R)$. This is because of the metric composition law for the directed landscape \eqref{eq: metric comp} (cf. the proof of \cite[Theorem 1.2]{sarkar2021brownian}). It thus suffices to prove the mutual absolute continuity of the law of the increments of
\[
\mathfrak{h}_1(y, h_0) = \max_{x\in \Q\setminus\{0\}}(h_0(x) + \mathcal{S}(x, y))\,,\quad y \in \R\,,
\]
against the rate two Wiener measure on some compact $K$. By translation symmetries of the Airy sheet, we can also without loss of generality take $K \subset (\varepsilon , y_0-\varepsilon)$ for some $y_0> 0$ and $\varepsilon\in (0, y_0)$.

More precisely, let $\mathcal{A}:\N\times \R\to \R$ be an Airy line ensemble and set $\tilde{\mathcal{A}}_1(x) = \mathcal{A}_1(-x), x\in \R$. Then, by the coupling between the Airy sheet and Airy line ensemble \eqref{eq: airy sheet coupling full space}, we can express for all $y\in \Q_+$,
\begin{align*}
\max_{x \in \mathbb{R}} \left( h_0(x) + \mathcal{S}(x, y) \right) &=\bigg(\max_{x \in (0, \infty)} \max_{\ell \geq 1} \left( h_0(x) + \mathcal A[x \to (0, \ell)] \right) + \mathcal A[(0, \ell) \to (y, 1)]\bigg) \\
&\vee \bigg(\max_{x \in (-\infty, 0)} \max_{\ell \geq 1} \left( h_0(x) + \tilde{\mathcal{A}}[x \to (-y_0, \ell)] + \tilde{\mathcal{A}}[(-y_0, \ell)\to (-y, 1)] \right)\bigg)\,.
\end{align*}
We thus have the alternative semi-discrete variational characterisation for the KPZ fixed point
\begin{align*}
\mathfrak{h}_1(y, h_0)&=\left( \max_{\ell \geq 1} G_{\ell}^{h_0} + \mathcal A[(0, \ell) \to (y, 1)] \right)\\
& \vee \left( \max_{\ell \geq 1} \tilde{G}_{\ell}^{h_0} + \tilde{\mathcal{A}}[(-y_0, \ell) \to (-y, 1)] \right)\,,\quad y \in \Q\cap [0, y_0]\,,
\end{align*}
where for $\ell \ge 1$,
\[
G_{\ell}^{h_0} = \max_{x \in (0, \infty)} \left( h_0(x) + \lim_{k\to \infty}\mathcal{A}[x_k \to (0, \ell)] - \mathcal{A}[x_k \to (0, 1)] + \mathcal{S}(x, 0) \right)\]
and
\[
\tilde{G}_{\ell}^{h_0} = \max_{x \in (-\infty, 0)} \left( h_0(x) + \lim_{k\to \infty}\tilde{\mathcal{A}}[(-x)_k \to (-y_0, \ell)] - \tilde{\mathcal{A}}[(-x)_k \to (-y_0, 1)] + \mathcal{S}(x, y_0) \right)\,,\]
where $x_k = (-\sqrt{k/(2x)}, k)$ for $x> 0$. Note by the ergodic properties of the Airy line ensemble and the above coupling with the Airy sheet, in particular, we have from \eqref{eq: ergod airy sheet} that for all $\ell \ge 1$, $G_{\ell}^{h_0}$ are $\mathscr{F}_- = \sigma(\{\mathcal{A}_i(x): x\le 0\})$ measurable
and $\tilde{G}_{\ell}^{h_0}$ are $\mathscr{F}_{(y_0, \infty)} \equiv \sigma(\{\mathcal{A}_i(x): x\ge y_0\})$ measurable since $\mathcal{S}(x, 0)$, $x> 0$ and $\mathcal{S}(x, y_0)$, $x<  0$ are $\mathscr{F}_-$ and $\mathscr{F}_{(y_0, \infty)}$-measurable respectively (cf. \eqref{eq: ergod airy sheet}).

We can now express using the metric composition law, Lemma \eqref{Lemma: Metric Composition}, almost surely (recall the notation for the top line of a Pitman transform, \eqref{eq: pitmantrans}),
\begin{align}\label{eq: pitman melon fin init KPZ}
    \mathfrak{h}_1(y, h_0)&= WF_1 (y) \lor WG_1(y)\,,\quad y \in \Q\cap[0, y_0]\,,
\end{align}
with the environments
\[
F = (F_1, F_1) = \bigg(G^{h_0}_1+\mathcal{A}_1(\cdot)-\mathcal{A}_1(0),  \max_{\ell \geq 2} (G_{\ell}^{h_0} + \mathcal A[(0, \ell) \to (\cdot, 2)]\bigg) \,, \quad \mbox{ and}
\]

\[
G = (G_1, G_1) = \bigg(\tilde{G}^{h_0}_1+\mathcal{A}_1(\cdot)-\mathcal{A}_1(y_0),  \max_{\ell \geq 2} (\tilde{G}_{\ell}^{h_0} + \tilde{\mathcal{A}}[(-y_0, \ell) \to (-\cdot, 2)])\bigg) \,.
\]

One can show that these two environments are actually continuous. Indeed, by the shape estimates for the Airy sheet, \eqref{eq: airy shape bnds} one can show by arguing similarly as in the proof of Theorem~ \ref{thm: mut abs cont unbounded supp init data} leading up to \eqref{eq: inf supp ef trunc}, that there exists $m^*\in \N$ almost surely finite such that
\begin{equation}\label{eq: bdry 1}
    \max_{\ell \geq 2} (G_{\ell}^{h_0} + \mathcal A[(0, \ell) \to (\cdot, 2)] = \max_{\ell\ge 2} (G_{\ell}^{h_0, m^*} + \mathcal A[(0, \ell) \to (\cdot, 2)])\,,
\end{equation}
and
\begin{equation}\label{eq: bdry 2}
    \max_{\ell \geq 2} (\tilde{G}_{\ell}^{h_0} + \tilde{\mathcal{A}}[(-y_0, \ell) \to (-\cdot, 2)]) = \max_{\ell\ge 2} (\tilde{G}_{\ell}^{h_0, m^*} + \tilde{\mathcal{A}}[(-y_0, \ell) \to (-\cdot, 2)])
\end{equation}

where for $m\ge 1$,
\[
G_{\ell}^{h_0, m} = \max_{x \in (0, m)} \left( h_0(x) + \lim_{k\to \infty}\mathcal{A}[x_k \to (0, \ell)] - \mathcal{A}[x_k \to (0, 1)] + \mathcal{S}(x, 0) \right)\]
and
\[
\tilde{G}_{\ell}^{h_0, m} = \max_{x \in (-m, 0)} \left( h_0(x) + \lim_{k\to \infty}\tilde{\mathcal{A}}[(-x)_k \to (-y_0, \ell)] - \tilde{\mathcal{A}}[(-x)_k \to (-y_0, 1)] + \mathcal{S}(x, y_0) \right)\,,\]
then continuity follows from Lemma \ref{lemma: infinite lpp cont at zero}. Moreover, the boundary data \eqref{eq: bdry 1}, \eqref{eq: bdry 2} are both 
{
 \begin{equation}\label{eq: gbibbs sigma alg fin mut}
  \mathscr{G} \equiv \sigma(\{\mathcal{A}_i(x): i\ge1\,, x\le 0 \text{ or } x\ge y_0\})\,
 \end{equation}
 }
measurable (by simply taking the minimal $m  =m^*\in \N$ so that the maxima are attained). Observe that for all $\ell \ge 1$ we have that the boundary data $G^{h_0, m}_\ell, \tilde{G}^{h_0, m}_\ell$ are strictly monotone by Proposition \ref{prop: separation of initial data} and the flip symmetry of the Airy line ensemble.

For ease of notation, we set $K = [a, b]$ with $0< a< b < y_0-\varepsilon$. Now, suppose for some $A\subseteq \mathscr{C}_0([a, b]; \R)$ Borel that
$\PP(\mathfrak{h}_1(\cdot)-\mathfrak{h}_1(a)\in A) = 0$.
Then, we estimate by inclusion and \eqref{eq: pitman melon fin init KPZ},
\begin{align}\label{eq: prob 1}
    0= \PP\bigg(&\mathcal{A}_1(\cdot) -\mathcal{A}_1(a) \in A\,,G^{h_0}_1 + \mathcal{A}_1(s)- \mathcal{A}_1(0) \ge F_2(s)\,, s\in [0, y_0-\varepsilon]\,,\nonumber\\
    &\tilde{G}^{h_0}_1+\mathcal{A}_1(s)-\mathcal{A}_1(y_0)\ge G_2(s)\,, s\in [a, y_0]\bigg)\,.
\end{align}

Now, by the Brownian bridge property \ref{subsec: Airy line ensemble} applied to $\mathscr{G}=$~\eqref{eq: gbibbs sigma alg fin mut} and standard properties of Brownian bridges, we can resample the top line of the Airy line ensemble as
\[
\mathcal{A}_1(s) = \begin{cases}
    &W_1(s) + L_1(s)\,,\quad s\in [0, a]\\
    &W_2(s) + L_2(s)\,,\quad s\in [a, y_0-\varepsilon]\\
    &W_3(s) + L_3(s)\,,\quad s\in [y_0-\varepsilon, y_0]\,,
\end{cases}
\]
conditioned on the event that that latter does not hit $\mathcal{A}_2$ on $[0, y_0]$, where $W_1, W_2, W_3$ are three mutually independent (also independent from $\mathcal{A}$) rate two Brownian bridges vanishing at both endpoints and $L_1, L_2, L_3$ are three affine functions with
\[
L_1(0) = \mathcal{A}_1(0)\,, L_1(a) = L_2(a) = \mathcal{A}_1(a)\,,\]
\[
L_2(y_0-\varepsilon) = L_3(y_0-\varepsilon) = \mathcal{A}_1(y_0-\varepsilon)\,, L_3(y_0) = \mathcal{A}_1(y_0)\,,
\]
We can thus express \eqref{eq: prob 1} as
\begin{align*}
    0&= \PP\bigg (W_2(s) + L_2(s) - \mathcal{A}_1(a) \in A\,,\\
    &G^{h_0}_1 + W_1(s) + L_1(s)- \mathcal{A}_1(0) \ge F_2(s)\,, s\in [0, a]\,, \\
    &W_1(s) + L_1(s)\ge \mathcal{A}_2(s)\,, s\in [0, a]\,,\\
    &G^{h_0}_1 + W_2(s) + L_2(s)- \mathcal{A}_1(0) \ge F_2(s)\,, s\in [a, y_0-\varepsilon]\,,\\
    &\tilde{G}^{h_0}_1+W_2(s) + L_2(s)-\mathcal{A}_1(y_0)\ge G_2(s)\,, s\in [a, y_0-\varepsilon])\\
    &W_2(s) + L_2(s)\ge \mathcal{A}_2(s)\,, s\in [a, y_0-\varepsilon]\,,\\
    &\tilde{G}^{h_0}_1+W_3(s) + L_3(s)-\mathcal{A}_1(y_0)\ge G_2(s)\,, s\in [y_0-\varepsilon, y_0])\\
    &W_3(s) + L_3(s)\ge \mathcal{A}_2(s)\,, s\in [y_0-\varepsilon, y_0]\bigg)\,.
\end{align*}

\begin{figure}
  \centering
       \includegraphics[width = 0.5\textwidth]{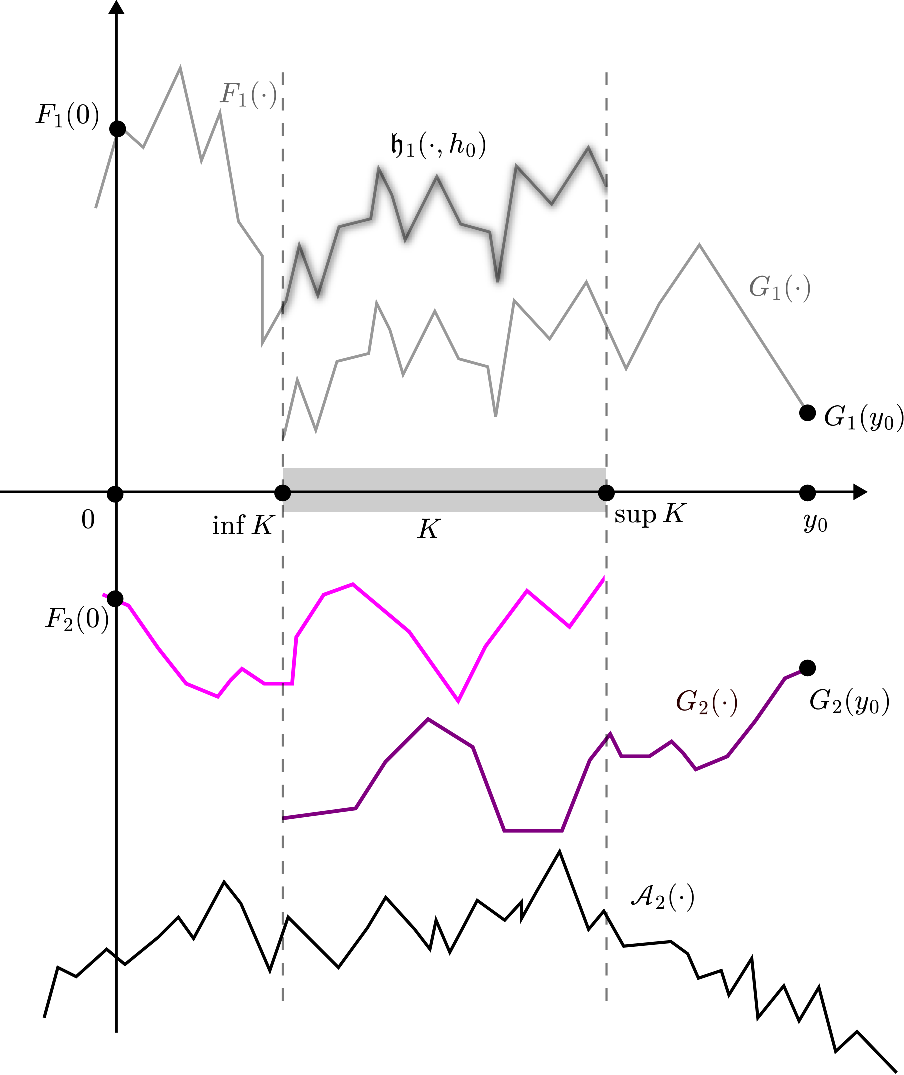}
       \caption{Illustration of the KPZ fixed point at unit time, $\mathfrak{h}_1(\cdot, h_0)$ on the compact interval $[0, y_0]$ in Theorem~\ref{thm: mut abs cont finitary init data} on the event the top lines of the Pitman transforms in \eqref{eq: pitman melon fin init KPZ}, $F_1(\cdot) \lor G_1(\cdot)$ does not hit $F_2(\cdot) \lor G_2(\cdot)$. On this event, $\mathfrak{h}_1(\cdot, h_0)$ is exactly equal to the Airy$_2$ process up to a height shift that is $\mathscr{G}$-measurable (recall the notation from Theorem~ \ref{thm: mut abs cont finitary init data}.
       Moreover, the `lower barrier' $F_2(\cdot) \lor G_2(\cdot)$ is $\mathscr{G}$-measurable. Using the Brownian Gibbs property, one can represent the top line of the Airy line ensemble (up to mutual absolute continuity) as a `Brownian bridge sandwich' (conditioned to avoid $\mathcal{A}_2$), that is a Brownian bridge starting from $0$ concatenated to a Brownian motion starting from $a$ which itself is concatenated to another Brownian bridge starting from $b$ and ending at $y_0$.}
       \label{fig: KPZ finitary melon}
   \end{figure} 

Since $b < y_0-\varepsilon$, by \cite[Lemma 3.14]{tassopoulos2025quantitativebrownianregularitykpz} the law of $W_2+L_2$ on $[a, b]$ is mutually absolutely continuous with respect to $B_2+\mathcal{A}_1(a)$ where $B_2$ is a rate two Brownian motion starting from $(a, 0)$ (independent from all the other randomness). Again by standard Brownian bridge properties, we can thus decompose 
\[
W_2+L_2(s) = \begin{cases}
    &B_2(s) + \mathcal{A}_1(a)\,,\quad s\in [a , b]\\
    &W'_2(s) + L'_2(s)\,,\quad s\in [b , y_0-\varepsilon]
\end{cases}
\]
where $W'_2$ is a rate two Brownian bridges vanishing at both endpoints (independent from all other randomness) and $L'_2$ is an affine function with $L'_2(b) = \mathcal{A}_1(a)+B_2(b)$ and $L'_2(y_0-\varepsilon) = \mathcal{A}_1(y_0-\varepsilon)$. We can now write \eqref{eq: prob 1} as
\begin{align*}
    0&= \PP\bigg (B_2(\cdot) \in A\,,B_2(s)+\mathcal{A}_1(a)\ge \mathcal{A}_2(s)\,, s\in [a, b]\,,\\
    &G^{h_0}_1 + W_1(s) + L_1(s)- \mathcal{A}_1(0) \ge F_2(s)\,, s\in [0, a]\,, \\
    &W_1(s) + L_1(s)\ge \mathcal{A}_2(s)\,, s\in [0, a]\,,\\
    &G^{h_0}_1 + B_2(s) + \mathcal{A}_1(a) - \mathcal{A}_1(0) \ge F_2(s)\,, s\in [a, b]\,,\\
    &G^{h_0}_1 + W'_2(s) + L'_2(s)- \mathcal{A}_1(0) \ge F_2(s)\,, s\in [b,y_0-\varepsilon]\,,\\
    &\tilde{G}^{h_0}_1+W_2(s) + L_2(s)-\mathcal{A}_1(y_0)\ge G_2(s)\,, s\in [a, y_0-\varepsilon]\\
    &W'_2(s) + L'_2(s)\ge \mathcal{A}_2(s)\,, s\in [b, y_0-\varepsilon]\,,\\
    &\tilde{G}^{h_0}_1+W_3(s) + L_3(s)-\mathcal{A}_1(y_0)\ge G_2(s)\,, s\in [y_0-\varepsilon, y_0])\\
    &W_3(s) + L_3(s)\ge \mathcal{A}_2(s)\,, s\in [y_0-\varepsilon, y_0]\bigg)\,.
\end{align*}

Since $L_1, L_2, L'_2, L_3$ are affine, their global minima on any interval are the respective minima of values at the endpoints thereof. 

Hence, we have by inclusion
\begin{align*}
    0&= \PP\bigg (B_2(\cdot) \in A\,,B_2(s)+\mathcal{A}_1(a)\ge \mathcal{A}_2(s)\,, s\in [a, b]\,,\\
    &G^{h_0}_1 + W_1(s) + L_1(s)- \mathcal{A}_1(0) \ge F_2(s)\,, s\in [0, a]\,, \\
    &W_1(s) + L_1(s)\ge \mathcal{A}_2(s)\,, s\in [0, a]\,,\\
    &G^{h_0}_1 + B_2(s) + \mathcal{A}_1(a) - \mathcal{A}_1(0) \ge F_2(s)\,, s\in [a, b]\,,\\
    &G^{h_0}_1 + W'_2(s) + B_2(b)+\mathcal{A}_1(a)\land \mathcal{A}_1(y_0-\varepsilon)- \mathcal{A}_1(0) \ge F_2(s)\,, s\in [b,y_0-\varepsilon]\,,\\
    &\tilde{G}^{h_0}_1+W_2(s) + \mathcal{A}_1(a)\land \mathcal{A}_1(y_0-\varepsilon)-\mathcal{A}_1(y_0)\ge G_2(s)\,, s\in [a, y_0-\varepsilon]\,,\\
    &W'_2(s) + B_2(b)+\mathcal{A}_1(a)\land \mathcal{A}_1(y_0-\varepsilon)\ge \mathcal{A}_2(s)\,, s\in [b, y_0-\varepsilon]\,,\\
    &\tilde{G}^{h_0}_1+W_3(s) + L_3(s)-\mathcal{A}_1(y_0)\ge G_2(s)\,, s\in [y_0-\varepsilon, y_0])\\
    &W_3(s) + L_3(s)\ge \mathcal{A}_2(s)\,, s\in [y_0-\varepsilon, y_0]\bigg)\,.
\end{align*}

Thus, by similar stochastic domination arguments as in the proof of Theorem~ \ref{thm: mutual abs cont} (see Figure \ref{fig: KPZ finitary melon} for an illustration), this time applied to $\mathcal{A}_1(a)$ and $\mathcal{A}_1(y_0-\varepsilon)$ jointly (conditioning on $\mathscr{G}$ using the Brownian Gibbs property \ref{subsec: Airy line ensemble}), we obtain
$\PP(B_2(\cdot) \in A)=0$
concluding the proof.
\end{proof}

We now record the following corollary that will be useful later.

\begin{corollary}\label{cor: KPZ fixed point mut abs cont prod}
    For $t> 0, a\in \R, K\subset \R$ with $a < \inf K$ bounded and $h_0 \in \mathscr{I}_t$, we have that there exists some probability measure on $\R$, $\eta$, such that
\[
\eta \otimes \mu\ll \mathrm{ Law \, of }\;(\mathfrak{h}_t(a, h_0), \mathfrak{h}_t(\cdot, h_0)-\mathfrak{h}_t(a, h_0)) \ll \mathrm{Leb} \otimes \mu\,,
\]
where $\mu$ denotes the law of a rate two Brownian motion starting at $(a, 0)$, restricted to $K$ and $\mathrm{Leb}$ the Lebesgue measure on $\R$.
\end{corollary}

\begin{proof}
    First, by translation symmetries of the Airy sheet \eqref{eq: Airy sheet symmetry}, we can take both $a, \inf K > 0$, with $B$ a rate two Brownian motion starting at $(a, 0)$ (and independent from $\mathfrak{h}(a)$). By the metric composition law for the directed landscape, \eqref{eq: metric comp} and independence, we can without loss of generality take $h_0$ to be continuous. 

    Inspecting the proof of Theorem~ \ref{thm: mut abs cont finitary init data} with $[a, \sup K ]$ in place of $K$ and conditioning on 
    \[
    \mathscr{G} \equiv \sigma(\{\mathcal{A}_i(x): (i, x) \not\in \{2, \ldots\}\times \R\cup \{1\}\times(0, \sup K +1)\})
    \]
    instead of \eqref{eq: gbibbs sigma alg fin mut} we obtain with $\mathfrak{h}_t(\cdot) \equiv \mathfrak{h}_t(\cdot, h_0)$,
    $\eta \otimes \mu\ll \mathrm{ Law \, of }\;(\mathfrak{h}_t(a), \mathfrak{h}_t(\cdot)-\mathfrak{h}_t(a))$,
    for some probability measure on $\R$, $\eta$.
    
    For the converse absolute continuity relation, observe that by a localisation argument for the support of $h_0$ (using the shape bounds for the Airy sheet \eqref{eq: airy shape bnds}), for any Borel $A$,
    \[
    \PP((\mathfrak{h}_t(a), \mathfrak{h}_t(\cdot)-\mathfrak{h}_t(a))\in A)\le \sum_{n=1}^\infty \PP((\mathfrak{h}^n_t(a), \mathfrak{h}^n_t(\cdot)-\mathfrak{h}^n_t(a))\in A)\,,
    \]
    where $\mathfrak{h}^n_t$ denotes the KPZ fixed point started from initial data $h_0\cdot \delta_{[-n, n]}$ (recall \eqref{eq: max-plus indicator}). Now, for any fixed $n$, by the translation symmetries of the Airy sheet \eqref{eq: Airy sheet symmetry}, it suffices to prove that
    \[
     \mathrm{ Law \, of }\;(\mathfrak{h}^n_t(a), \mathfrak{h}^n_t(\cdot)-\mathfrak{h}^n_t(a))
     \ll \mathrm{Leb} \otimes \mu \,,
    \]
    for $a = 0$,  $K \subseteq (0, \infty)$ compact. In this case, the coupling \ref{eq: airy sheet coupling full space} between the Airy sheet and Airy line ensemble gives
    \[
    \mathfrak{h}_t^n (y) = \max_{\substack{1\le \ell\le L_0 }} (G_\ell + \mathcal{A}[(0,\ell)\to (y, 1)])\,,\quad y \in [0, \sup K+1]\,,
    \]
    for some almost surely finite random $L_0$. 
    
    It thus suffices to show that the laws of
    \begin{equation}\label{e:abstoshows}
        (G_1, \max_{\substack{1\le \ell\le m }} (G_\ell + \mathcal{A}[(0,\ell)\to (\cdot, 1)]) \ll \mathrm{Leb}\otimes \mu 
    \end{equation}
    on $\mathscr{C}(K; \R)$ for any $m\ge 1$, by a similar localisation argument as above, sectioning on events $\{L_0\le m \}$. Clearly $G_1 = \max(h_0(x) + \mathcal{S}(x, 0)) = \mathfrak{h}_1(0)\ll \mathrm{Leb}$. Now \eqref{e:abstoshows} follows from the Brownian Gibbs property, \ref{subsec: Airy line ensemble} conditioning on
    $\mathscr{F}_- = \sigma(\{\mathcal{A}_i(x): x\le 0\})$,
    using the fact that $G_\ell$ are $\mathscr{F}_-$ measurable and then applying \cite[Theorem~ 7.9]{tassopoulos2025inhomogeneousbrownian} to the iterated Skorokhod reflections (cf. \ref{sec: prelims}) with boundary data $(G_\ell)_{1\le \ell \le m}$ on $K$. 
\end{proof}

The mutual absolute continuity of the increments of the KPZ fixed point against Brownian motion on compacts gives the former has `full topological support'. This is the content of the following corollary.

\begin{corollary}\label{cor: support kpz}
    Let $K\subseteq\R$ be a bounded interval, $t>0$, $h_0\in \mathscr{I}_t$ and $f\in \mathscr{C}_0(K; \R)$. Then for any $\varepsilon>0$, we have
    $\PP(\norm{f-\mathfrak{h}_t(\cdot, h_0)+\mathfrak{h}_t(\inf K, h_0)}_{L^\infty(K)} < \varepsilon) > 0$.
\end{corollary}
\begin{proof}
    The proof of both parts follows from the mutual absolute continuity relation in Theorem \ref{thm: mut abs cont finitary init data} and Lemmas \ref{lemma: bb comparison lemma} and Lemma \ref{lemma: support bb}.
\end{proof}

\section{The Airy Sheet and additive Brownian motion}\label{sec: top supp airy sheet}

In this section, we prove mutual absolute continuity of additive Brownian motion to the Airy sheet on compacts. 

\begin{figure}[H]
    \centering
    \includegraphics[width=\linewidth]{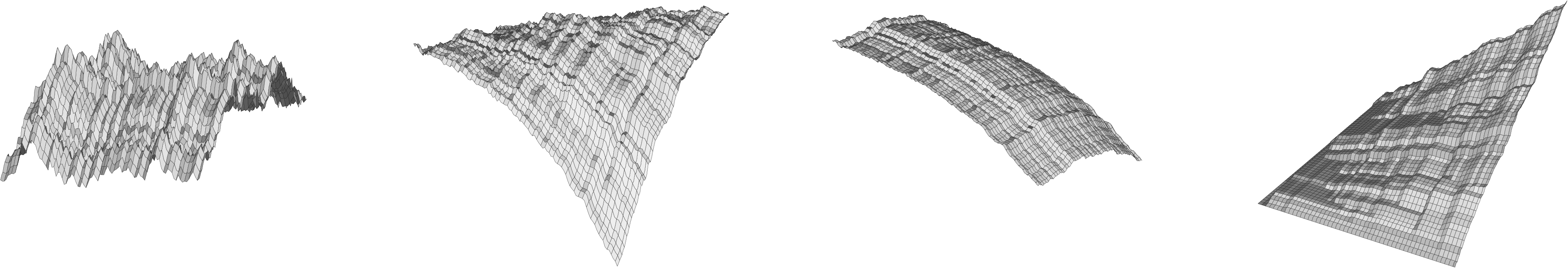}
    \caption{Illustrations of: \textbf{left}: Additive Brownian motion, \textbf{centre-left}: $\mathcal{S}(\cdot, \cdot')$, \textbf{centre-right}: $\mathcal{S}(\cdot, 0) + \mathcal{S}(0, \cdot')$ and \textbf{right}: $\mathcal{S}(\cdot, \cdot') -\mathcal{S}(\cdot, 0) - \mathcal{S}(0, \cdot')+\mathcal{S}(0, 0)$ on $[0, 1]^2$.}
    \label{fig: airy sheet}
\end{figure}

Observe the quadrangle inequality \ref{eq: Airy sheet monoton}
\begin{equation*}
    \mathcal{S}(x,y)+\mathcal{S}(x',y')\ge \mathcal{S}(x,y')+\mathcal{S}(x',y)\,,\quad x\le x',\; y\le y' 
\end{equation*}
gives for any $x_0, y_0$, the random continuous function
\[
    \mathcal{S}(x,y)-\mathcal{S}(x, y_0) - \mathcal{S}(x_0, y) + \mathcal{S}(x_0, y_0)\,,\quad x\ge x_0\,, y\ge y_0
    \]
    is monotone in $x, y$ and is thus the cumulative distribution function of a random Borel measure $\mu_{x_0, y_0}$ on $[x_0, \infty)\times [y_0, \infty)$ (cf. \textbf{right} in Figure \ref{fig: airy sheet}). We can thus write
  \begin{equation}\label{eq: airy diff}
       \mu_{x_0, y_0}([x_0, x]\times [y_0, y]) = \mathcal{S}(x,y)-\mathcal{S}(x, y_0) - \mathcal{S}(x_0, y) + \mathcal{S}(x_0, y_0)\,.
  \end{equation}
  We start with a lemma that states with probability strictly between zero and one, the quadrangle inequality \eqref{eq: Airy sheet monoton},
becomes an equality on compacts (cf. \textbf{centre-right} in Figure \ref{fig: airy sheet}), or in other words the measures $\mu_{x_0, y_0}$ are not degenerate but can vanish on compacts with positive probability. 

\begin{proposition}\label{prop: Airy sheet quadrangle equality}
    Fix $M>0$, then we have 
    \begin{equation}\label{eq: prob zero one}
    \PP(\mathcal{S}(x, y) = \mathcal{S}(x, -M) + \mathcal{S}(-M, y) - \mathcal{S}(-M, -M)\,,\text{ for all } |x|+ |y|\le M)>0\,.
    \end{equation}
    Moreover,
     \begin{equation}\label{eq: conv to zero prob}
        \lim_{m\to \infty}\PP(\mathcal{S}(x, y) = \mathcal{S}(x, -M) + \mathcal{S}(-M, y) - \mathcal{S}(-M, -M)\,,\text{ for all } |x|+ |y|\le M) = 0\,.
    \end{equation}
\end{proposition}

\begin{remark}
    This result also has a geometric interpretation in terms of geodesics in the directed landscape (cf. \cite[Section 12]{DOV} for a definition). In particular, by \cite[Lemma 3.15]{Gangulyetal2022}, the above event is equivalent to the fact that no two directed geodesics with endpoints $(x_1, 0), (y_1, 1)$ and $(x_2, 0), (y_2, 1)$ for $-M \le x_1 <  x_2 \le M, -M \le y_1 < y_2 \le M$, remain disjoint on $[0,1]$.
\end{remark}

\begin{proof}
    Observe that by shift invariance of the Airy sheet, \eqref{eq: Airy sheet symmetry} we have 
    \[
    \PP(\mathcal{S}(x, y) = \mathcal{S}(x, -M) + \mathcal{S}(-M, y) - \mathcal{S}(-M, -M)\,,\text{ for all } |x|+ |y|\le M)\]
    \[
    =\PP(\mathcal{S}(x, y) = \mathcal{S}(x, 0) + \mathcal{S}(0, y) - \mathcal{S}(0, 0)\,,\text{ for all } (x, y) \in [0, 2M]^2) \,.
    \]

    Now, we have using the coupling between the Airy sheet and Airy line ensemble and Lemma \ref{lemma: infinite lpp cont at zero} that almost surely $\mathcal{S}(0, \cdot) = \mathcal{A}_1(\cdot)$ and for all $(x, y) \in [0, 2M]^2\cap \Q^2$,
    \[
    \mathcal{S}(x, y) = \max_{1\le \ell\le L_M}(\mathcal{A}[x\to (0, \ell)]+ \mathcal{A}[(0, \ell)\to (y, 1)])\,,
    \]
    where $L_M\ge 2$ is an almost surely finite random constant depending only on $M$ (and can be extracted from the geodesic geometry of the Airy line ensemble).

    Moreover, for all $x\in (0, 2M]\cap \Q$, one has the uniform almost sure bounds due to geodesic geometry in the Airy line ensemble, Proposition \ref{prop: separation of initial data}
    \begin{align}\label{eq: gap unif airy lpp}
    \eta_M &:= \sup_{x\in (0, 2M)\cap \Q}(\mathcal{A}[x\to (0, 2)]-\mathcal{A}[x\to (0, 1)])\\
    &\le \mathcal{A}[(\varepsilon^\infty_{\lceil 2M\rceil , 2}, 2)\to (0, 1)]-\mathcal{A}[(\varepsilon^\infty_{\lceil 2M\rceil , 2}, 2)\to (0,2)]
    \end{align}
    where $\varepsilon^\infty_{\lceil 2M\rceil , 2}$ is the first time the semi-infinite geodesic $\pi[\lceil 2M\rceil  \to (0, 2)]$ (cf. \ref{def: semi-inf geo}) reaches level $2$ in the environment given by the Airy line ensemble. The fact that this upper bound is indeed almost surely strictly less than zero follows from \cite[Lemma 4.7]{tassopoulos2025quantitativebrownianregularitykpz} (essentially from the locally Brownian nature of the Airy line ensemble and Lemma \ref{lemma: semi-inf geod jump time pos} in the Appendix).

    Now, we can represent the Airy sheet as the top line of the melon (cf. \eqref{eq: pitmantrans})
    \[
    \mathcal{S}(x, y) = WF^x_1(y)\,, (x, y) \in (0, 2M) \cap \Q^2\,,
    \]
    where the random environment $F^x=(F_1^x,F_2^x)\in \mathscr{C}^2([0, 2M] ;\R)$ is given by
    \begin{align*}
        (F^x_1, F^x_2)&=(\mathcal{A}[x\to (0, 1)]+\mathcal{A}_1(\cdot) -\mathcal{A}_1(0), \mathcal{A}[x\to (\cdot, 2)])\,,
    \end{align*}
    {
    where
    \begin{align*}
        \mathcal{A}[x\to (\cdot, 2)] &:= \lim_{k\to \infty}(\mathcal{A}[x_k\to (\cdot, 2)]-\mathcal{A}[x_k \to (\cdot, 1)]+ \mathcal{S}(x, \cdot))\\
        &=\max_{2\le \ell\le L_M}(\mathcal{A}[x\to (0, \ell)]+ \mathcal{A}[(0, \ell)\to (\cdot, 2)])\,.
    \end{align*}
    }
    Thus, it suffices to show that with probability strictly between zero and one,
    \[
    \mathcal{S}(x, y) = WF^x_1(y) = \mathcal{A}[x\to (0, 1)]+\mathcal{A}_1(\cdot) -\mathcal{A}_1(0) = \mathcal{S}(x, 0) + \mathcal{S}(0, y) -\mathcal{S}(0, 0)\,, (x, y) \in (0, 2M) \cap \Q^2\,.
    \]
    This happens if and only if almost surely for all $x\in (0, 2M) \cap \Q$, $F^x_2\le F^x_1$ on $[0, 2M]$. In other words, if 
    \[
    \mathcal{A}[x\to (0, 1)]+\mathcal{A}_1(\cdot) -\mathcal{A}_1(0)\ge \mathcal{A}[x\to (\cdot, 2)]\,,\quad y \in [0, 2M]\,.
    \]
    Now, the uniform bound \eqref{eq: gap unif airy lpp}, the monotonicity of $\mathcal{A}[x\to (0, \ell)]$ (cf. Proposition \ref{prop: separation of initial data}) and last passage values implies we can estimate pointwise
   {
    \begin{align*}
        F^x_2(y) &= \mathcal{A}[x\to (y, 2)]\le \max_{2\le \ell\le L_M}(\mathcal{A}[x\to (0, \ell)]+ \mathcal{A}[(0, \ell)\to (y, 2)])\\
        &\le \mathcal{A}[x\to (0, 2)]+ \mathcal{A}[(0, L_M)\to (y, 2)])\\
        &= \mathcal{A}[x\to (0, 1)] + \eta_M + \mathcal{A}[(0, L_M)\to (y, 2)]\,,\quad y \in [0, 2M]\,.
    \end{align*}
    }

    Note by ergodic properties of the Airy sheet, \eqref{eq: ergod airy sheet} and the coupling \ref{eq: airy sheet coupling full space}, $\mathcal{A}[x\to (0, \ell)]$ is $\sigma(\{\mathcal{A}_i(x): (i, x) \in (-\infty, 0)\times \N\})$-measurable (cf. \cite[Lemma 3.9]{sarkar2021brownian}), and thus is $F^x_2$ is
    \[
    \sigma(\{\mathcal{A}_i(x): (i, x) \in \N\times (-\infty, 0) \cup  \{2, \ldots\}\times \R\})\]
    measurable. Moreover, by the Brownian Gibbs property for the Airy line ensemble and standard facts about Brownian bridges, we have 
    \[
    \mathcal{A}_1(y) -\mathcal{A}_1(0)\ge \eta_M + \mathcal{A}[(0, L_M)\to (y, 2)])\,,\quad y \in [0, 2M]
    \]
    with positive probability. 
    
    To show \eqref{eq: conv to zero prob}, it suffices to observe that with the remark above, that two geodesics in the directed landscape with endpoints $(x_1, 0), (y_1, 1)$ and $(x_2, 0), (y_2, 1)$ for $|x_1 - x_2|\land |y_1-y_2| \gg 1$ remain disjoint on $[0,1]$ with probability converging to one. Indeed, this follows from \cite[Theorem 1.7]{DOV}, which states that any directed geodesic concentrates near the line passing through its endpoints. This gives $\mu([-M, -M+N]\times[-M, -M+ N])>0$ with probability tending to one as $N \to \infty$, concluding the proof.
\end{proof}

We arrive at the following absolute continuity result. It states that the law of the Airy sheet (up to centering) \emph{majorises} the law of \emph{additive Brownian motion}, that is the sum of two independent rate two Brownian motions, on compacts. This is the content of the following theorem. 

\begin{theorem}\label{thm: abs cont airy sheet}
    Fix $K\subset \R^2$ compact. Then, with $B, B'$ two independent Brownian motions starting on $(\inf K , 0 )$, the law of $B(\cdot) + B'(\cdot')$
    on $\mathscr{C}_0(K; \R)$ is absolutely continuous with respect to the law of the Airy sheet $\mathcal{S}(\cdot, \cdot')-\mathcal{S}(\inf K, \inf K)$ on $\mathscr{C}_0(K; \R)$.
\end{theorem}
\begin{proof}
 By translation symmetries of the Airy sheet, we can replace $K$ with $[1, 2\mathrm{diam} K+1]^2$.

    Let $A\subseteq \mathscr{C}_0(K; \R)$ be Borel measurable. Using the coupling of the Airy sheet with the Airy line ensemble, \ref{eq: airy sheet coupling full space}, we have with $M = \mathrm{diam} K$ on the event (recall $\eta_M$ from \eqref{eq: gap unif airy lpp} in Proposition \ref{prop: Airy sheet quadrangle equality})
\[
A_M \equiv \{\mathcal{A}_1(y) - \mathcal{A}_1(0)\ge \eta_M + \mathcal{A}[(0, L_M)\to (y, 2)]\,,\quad y \in [0, 2M+1]\}\,,
\]
\[
\mathcal{S}(x, y) = \mathcal{S}(0,y)-\mathcal{S}(0,0) + \mathcal{S}(x, 0) = \mathcal{A}_1(y)-\mathcal{A}_1(0) + \mathcal{S}(x, 0)\,.
\]
Note by the coupling \eqref{eq: airy sheet coupling full space} and ergodic properties of the Airy sheet \eqref{eq: ergod airy sheet}, $\mathcal{S}(\cdot, 0)$, $\eta_M$, $L_M$ are 
    \[
    \mathscr{F}_{(-\infty, 0)\times \N\cup \R\times \{2, \ldots\}}\equiv \sigma(\{\mathcal{A}_i(x): (i, x) \in \N\times (-\infty, 0) \cup  \{2, \ldots\}\times \R\})
    \]
    measurable. We can thus estimate
    \begin{align*}
        \PP(\mathcal{S}(\cdot, \cdot ')&-\mathcal{S}(\inf K, \inf K)\in A)\ge \PP( \mathcal{A}_1(\cdot)-\mathcal{A}_1(1) + \mathcal{S}(\cdot', 0)-\mathcal{S}(1, 0) \in A, A_M)\\
        &= \PP\big(\mathcal{A}_1(\cdot)-\mathcal{A}_1(1) + \mathcal{S}(\cdot', 0)-\mathcal{S}(1, 0) \in A,\\
        & \mathcal{A}_1(y)-\mathcal{A}_1(1)+\mathcal{A}_1(1)- \mathcal{A}_1(0)\ge \eta_M + \mathcal{A}[(0, L_M)\to (y, 2)])\,, y \in [0, 2M+1]\big)\,.
    \end{align*}
    
    Now, by the Brownian Gibbs property, (cf. \ref{subsec: Airy line ensemble}), one can resample the Airy line ensemble on $\{1\}\times [1, 2M+2]$ to have the law of the affine shift $L(\cdot)+W(\cdot)$ of an independent Brownian bridge $W(\cdot)$ starting from $(1, 0)$ and ending at $(2M+2, 0)$, with $L(0) = \mathcal{A}_1(1)$ and $L(2M+2) = \mathcal{A}_1(2M+2)$ conditioned to avoid $\mathcal{A}|_{[1, 2M+2]\times \{2\}}$. In particular, by Lemma \ref{lemma: bb comparison lemma}, we have the law of $L(\cdot) + W(\cdot)$ restricted to $[1, 2M+1]$ is mutually absolutely continuous with respect to the law of $\mathcal{A}_1(1) + B(\cdot)$, where $B$ is a Brownian motion starting at $(1, 0)$.  Thus, we have
    \begin{align*}
        &\PP\big(L(\cdot) + W(\cdot)-\mathcal{A}_1(1) + \mathcal{S}(\cdot', 0)-\mathcal{S}(1, 0) \in A,\\
        &\mathcal{A}_1(y)- \mathcal{A}_1(0)\ge \eta_M + \mathcal{A}[(0, L_M)\to (y, 2)])\,, y \in [0, 1]\,,\\
        & W(\cdot) + L(\cdot)-\mathcal{A}_1(1)+\mathcal{A}_1(1)- \mathcal{A}_1(0)\ge \eta_M + \mathcal{A}[(0, L_M)\to (y, 2)])\,, y \in [1, 2M+1]\,,\\
        & W(y) + L(y)\ge \mathcal{A}_2(y)\,, y \in [1, 2M+1]\big) = 0
    \end{align*}
    if and only if
    \begin{align*}
        &\PP\big(B(\cdot)+ \mathcal{S}(\cdot', 0)-\mathcal{S}(1, 0) \in A,\mathcal{A}_1(y)- \mathcal{A}_1(0)\ge \eta_M + \mathcal{A}[(0, L_M)\to (y, 2)])\,, y \in [0, 1]\,,\\
        & B(\cdot)+\mathcal{A}_1(1)- \mathcal{A}_1(0)\ge \eta_M + \mathcal{A}[(0, L_M)\to (y, 2)])\,, y \in [1, 2M+1]\,,\\
        &\mathcal{A}_1(1) + B(\cdot)\ge \mathcal{A}_2(y)\,, y \in [1, 2M+1]\big) = 0\,.
    \end{align*}
    Now, conditioning on $\mathscr{F}^\mathrm{ext}_{[1, 2M+1]\times \{1\}}\lor \sigma(B_t: t\in [1, 2M+1])$ where
    \[
    \mathscr{F}^\mathrm{ext}_{[1, 2M+1]\times \{1\}}\equiv \sigma(\{\mathcal{A}_i(x): (i,x)\not\in  \{1\}\times [1, 2M+1])
    \]
     by the Brownian bridge property \ref{subsec: Airy line ensemble}, stochastic domination arguments using Lemma \ref{lemma: bridge monotonicity} and \ref{lemma: support bb} (recall $\mathcal{S}(\cdot, 0)$, $\eta_M$, $L_M$ are $\mathscr{F}_{(-\infty, 0)\times \N\cup \R\times \{2, \ldots\}}$-measurable),
    \begin{align*}
        &\PP\bigg(\mathcal{A}_1(y)- \mathcal{A}_1(0)\ge \eta_M + \mathcal{A}[(0, L_M)\to (y, 2)])\,, y \in [0, 1]\,,\\
        & B(\cdot)+\mathcal{A}_1(1)- \mathcal{A}_1(0)\ge \eta_M + \mathcal{A}[(0, L_M)\to (y, 2)])\,, y \in [1, 2M+1]\,,\\
        & \mathcal{A}_1(1) + B(\cdot)\ge \mathcal{A}_2(y)\,, y \in [1, 2M+1]\bigg | \mathscr{F}^\mathrm{ext}_{[1, 2M+1]\times \{1\}}\lor \sigma(B_t: t\in [0, 1, 2M+1])\bigg) >0
    \end{align*}
    almost surely. We thus deduce that the law of $B(\cdot) + \mathcal{S}(\cdot', 0)$ (for $B$ and $\mathcal{S}$ independent)
    on $\mathscr{C}_0(K; \R)$ is absolutely continuous with respect to the law of the Airy sheet $\mathcal{S}(\cdot, \cdot')-\mathcal{S}(\inf K, \inf K)$ on $\mathscr{C}_0(K; \R)$. To conclude the proof observe that the law of $\mathcal{S}(\cdot, 0)-\mathcal{S}(\inf K , 0)$ is mutually absolutely continuous with respect to a rate two Brownian motion on $K$, by the coupling of the Airy sheet with the Airy line ensemble \eqref{eq: airy sheet coupling full space} and Theorem \ref{thm: mut abs cont airy bm}.
\end{proof}

This means the law of the Airy sheet has full topological support in the space of certain \emph{separable} continuous functions in $\R^2$. In particular, generalising the quadrangle equality in the statement of Proposition \ref{prop: Airy sheet quadrangle equality}, for $K\subset \R^2$ compact, we denote by
    \begin{equation} \label{eq: rect}
    \bm{\mathrm{Rect}}_K \equiv \{h\in \mathscr{C}^2(K; \R): h(x_1, y_1) + h(x_2, y_2)-h(x_1, y_2) -h(x_2, y_1) = 0\,, \text{ for all } x_{1, 2}, y_{1, 2}\in K\}\,,
    \end{equation}
    the set of all continuous functions on $K$ satisfying the so-called \emph{rectangle property} above.

Theorem \ref{thm: abs cont airy sheet} \emph{cannot} be strengthened to mutual absolute continuity between the Airy sheet and additive Brownian motion. This is indeed the case and is the content of the following proposition.

\begin{proposition}\label{prop: not mut abs cont airy sheet additive brownian motion}
    Fix $a, b \in \R$ and $B, B'$ two independent Brownian motions starting from $(a , 0 )$ and $(b, 0)$ respectively. Then, the law of the centred Airy sheet $\mathcal{S}(\cdot, \cdot')-\mathcal{S}(a, b)$ on $\mathscr{C}_0([a, \infty)\times [b, \infty); \R)$ 
    is \emph{not} absolutely continuous with respect to the law of the additive Brownian motion $B(\cdot) + B'(\cdot')$ on $\mathscr{C}_0([a, \infty)\times [b, \infty); \R)$.
\end{proposition}

\begin{proof}
    Suppose for a contradiction that the law of the Airy sheet $\mathcal{S}(\cdot, \cdot')-\mathcal{S}(a, b)$ on $\mathscr{C}_0([a, \infty)\times [b, \infty); \R)$ 
    \emph{is} absolutely continuous with respect to the law of the additive Brownian motion $B(\cdot) + B'(\cdot')$ on $\mathscr{C}_0([a, \infty)\times [b, \infty); \R)$. Consider for any $M> 1$, the continuous (with respect to the uniform topology) functional on $\mathscr{C}_0([a, \infty)\times [b, \infty); \R)$, 
    \[
    f\mapsto \Delta^M (f) \stackrel{\mathrm{def}}{=} \max_{(x, y)\in [a, a+M]\times [b, b+M]} (f(x, y) - 
    f(x, b ) -f(a, y) + f(a,b))\,.
    \]
    By Proposition \ref{prop: Airy sheet quadrangle equality}, the Airy sheet satisfies
    \[
    \Delta^M(\mathcal{S}) = \max_{(x, y)\in [a, a+M]\times [b, b+M]}\mu_{a, b}([a, x]\times [b, y]) > 0
    \]
    with positive probability for all $M>1$ sufficiently large, whereas the restriction of $\Delta^M$ to $\bm{\mathrm{Rect}}_{[a, a+M]\times [b, b+M]^2}$ is identically zero for all $M$. This is a contradiction as additive Brownian motion is supported on $\bm{\mathrm{Rect}}_{[a, a+M]\times [b, b+M]^2}$.
\end{proof}

We end this subsection with the following corollary regarding the topological support of the Airy sheet. 

\begin{corollary}\label{cor: Airy sheet top supp}
Fix $K\subset \R^2$ compact, $h\in \bm{\mathrm{Rect}}_K$ (cf. \eqref{eq: rect}) and $\varepsilon > 0$. Then, for any compact $K\subseteq \R^2$,
\[
\PP(||\mathcal{S}(\cdot, \cdot')-\mathcal{S}(\inf K , \inf K)-(h(\cdot, \cdot')-h(\inf K , \inf K))||_{L^\infty(K)}< \varepsilon) > 0\,.
\]
\end{corollary}
\begin{proof}
    The proof is a direct application of Theorem \ref{thm: abs cont airy sheet}, Lemma \ref{lemma: support bb} and the fact that any $h\in \bm{\mathrm{Rect}}_K$ satisfies $h(\cdot, \cdot') = h(0, \cdot)-h(0,0) + h(\cdot', 0)-h(0,0)$.
\end{proof}

\section{Applications}\label{sec: applications}

In this section, we discuss some applications of the mutual absolute continuity result of the KPZ fixed point established in Theorem \ref{thm: mut abs cont finitary init data} and the absolute continuity result of the Airy sheet proved in Theorem \ref{thm: abs cont airy sheet}. 

\subsection{Record times for the KPZ fixed point}

We obtain as a direct consequence of Theorem~ \ref{thm: mut abs cont finitary init data} a result involving the `topological support' of record times for the KPZ fixed point, that is the times the KPZ fixed point attains its running maximum (relative to some starting point). 

Define the random closed subset of \emph{record times} for the KPZ fixed point $\mathfrak{h}_t(\cdot, h_0)$ at time $t>0$ started from initial data $h_0\in \mathscr{I}_t$,
\[
\mathscr{R}(a, t, h_0) \stackrel{\mathrm{def}}{=} \{y\ge a: \mathfrak{h}_t(y, h_0) = \max_{a\le s\le y}\mathfrak{h}_t(s, h_0)\}\,, \quad a\in \R\,.
\]

In the following corollary, we obtain that in any interval $[b, c]$, $a< b< c$, the set of record times $\mathscr{R}(a, t, h_0)\cap [b, c]$ is non-empty with probability strictly between zero and one. 

\begin{corollary}\label{cor: record times}
    For $t> 0$, $a< b< c\in \R$ and $h_0\in \mathscr{I}_t$, 
    \[
    0 < \PP(\mathscr{R}(a, t, h_0)\cap [b, c] \neq \emptyset ) < 1\,.
    \]
\end{corollary}

\begin{proof}
    By the mutual absolute continuity of the increments of the KPZ fixed point against Brownian motion on compacts and Brownian scaling, we obtain that 
    $0 < \PP(\mathscr{R}(a, t, h_0)$ $\cap [b, c] \neq \emptyset ) < 1$
    if and only if
   $0 < \PP(\{y\ge a: W(y)= \max_{a\le s\le y}W(s)\}$ $\cap [b, c] \neq \emptyset ) < 1$
    where $W$ is a standard Brownian motion starting from $(a, 0)$. By L\'evy's theorem for reflected Brownian motion, see \cite[Theorem~ 2.34]{morters_peres_2010}, $\max_{a\le s\le \cdot}W(s)-W(a)$ has the same law on paths on $\mathscr{C}_0([a, \infty); \R)$ as $|W(\cdot)|$. Hence, we have
    \[
    \PP(\{y\ge a: W(y) = \max_{a\le s\le y}W(s)\}\cap [b, c] \neq \emptyset ) = \PP(\exists y\in [b, c] : W(y) = 0)\,.
    \]
    The latter is clearly seen to be strictly between zero and one, hence the result follows.
\end{proof}

\subsection{Hitting probabilities of the KPZ fixed point and capacity}\label{sec: hit prob KPZ fixed pt cap}

Let $E, F$ be compact subsets of $\R$. Fix $t> 0$ and let $h_0\in \mathscr{I}_t$ and recall the notation for the KPZ fixed point started from $h_0$ at time $t$, $\mathfrak{h}$ (suppressing time dependence). We are interested in giving necessary and sufficient conditions for the positivity of probabilities of the form
\begin{equation}\label{eq: prob non vanish hit}
\PP(\mathfrak{h}(E)\cap F \neq \emptyset) = \PP(\mathrm{Gr}(\mathfrak{h})\text{ hits } E\times F) > 0\,,
\end{equation}
where the graph of the KPZ fixed point is denoted by
$\mathrm{Gr}(\mathfrak{h}) \stackrel{\mathrm{def}
}{=} \{(t, \mathfrak{h}(t)): t\in \R\}\subseteq \R^2$.
In other words, we are interested in the probability the KPZ fixed point on some compact $E$ hits another compact $F$, see Figure \ref{fig: KPZ graph} below.
\begin{figure}[ht]
    \centering
    \includegraphics[width=0.6\textwidth]{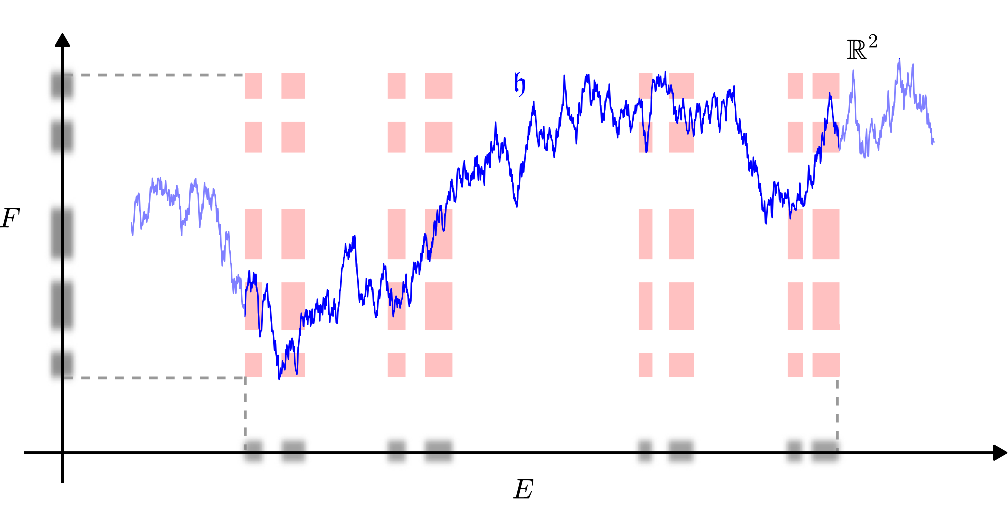}
    \caption{Illustration of the graph of the rescaled increments of the KPZ fixed point in \color{blue}blue \color{black} and the set $E\times F$ in \color{red}red\color{black}. The probability $\mathfrak{h}(E)\cap F \neq \emptyset$ if and only if the graph of the KPZ fixed point, $\mathrm{Gr}(\mathfrak{h})$ hits the red region with positive probability.}
    \label{fig: KPZ graph}
\end{figure}

We will in fact be able to obtain a characterisation for \eqref{eq: prob non vanish hit} by comparing \eqref{eq: prob non vanish hit} to Brownian hitting probabilities using the mutual absolute continuity of increments of the KPZ fixed point, Corollary \ref{cor: KPZ fixed point mut abs cont prod}. 

It is a well-known folklore fact that for a Brownian motion $W$, $W(E)$ intersects $F$ with
positive probability if and only if $E\times F$ has positive thermal
capacity in the sense of Watson \cite{watson1978corrigendum, watson1978thermal}. Now, by Corollary \ref{cor: KPZ fixed point mut abs cont prod} and analytic properties of tail probabilities of the KPZ fixed point, see \cite[Section 4]{quastelkpzfixedpoint2021} (which give a positive density against the Lebesgue measure on $\R$) we obtain the absolute continuity relations
\[
\mathrm{ Law \, of }\; (Y, B(\cdot))\ll \mathrm{ Law \, of }\; (\mathfrak{h}(a), \mathfrak{h}(\cdot)-\mathfrak{h}(a))\ll \mathrm{Leb}\otimes \mu\,,
\]
on $\R\times \mathscr{C}(K;\R)$, for any $K\subseteq \R$ compact and $a< \inf K$, with $B$ a rate two Brownian motion starting at $(a, 0)$ (and independent from the random variable $Y$), the above statements transfer verbatim to the KPZ fixed point and $E,F$ as above.

We first formulate the statement of the characterisation in the case where the Lebesgue measure of $F$ is zero, since otherwise the intersection probability
$\PP(\mathfrak{h}(E)\cap F\neq\emptyset)$ is strictly positive for every non-empty Borel set $E \subset \R$.

\begin{definition}\label{def: thermal cap}
Let $E \subset \R$ and $F \subset \R$ be compact sets.
For $\gamma \ge 0$, the $\gamma$-thermal capacity of the set $E\times F$ is defined by
\begin{equation}\label{def:TC}
C_\gamma(E \times F) = \frac{1}{\inf\{ \mathscr{E}_\gamma(\mu) : \mu\in \mathscr{P}(E\times F)\}},
\end{equation}
where $\mathscr{P}(E\times F)$ denotes the set of all Borel probability measures $\mu$ on $E\times F$ such that $\mu(\{t\} \times F) = 0$ 
for all $t\in E$ and $\mathscr{E}_\gamma(\mu)$ is the $\gamma$-thermal energy of $\mu$ defined by 
\begin{equation}\label{eq: thermal cap}
\mathscr{E}_\gamma(\mu)\stackrel{\mathrm{def}}{=} \int_{E\times F}\int_{E\times F}\frac{\mathrm{e}^{-|x-y|^2/(4|t-s|)}}{
|t-s|^{1/2}\cdot |x-y|^\gamma} \mu(\diff
s \,\diff x) \mu(\diff t \,\diff y)\,.
\end{equation}

\end{definition}

\begin{corollary}\label{cor: thermal cap}
Suppose $F \subset\R$ is compact and has Lebesgue measure 0. Then 
$\PP(\mathfrak{h}(E)\cap F\neq\emptyset)>0$
if and only if $\mathscr{C}_0(E\times F) > 0$.
\end{corollary}

\begin{proof}
    First assume $\mathscr{C}_0(E\times F) > 0$. Then, observe for any $a < \inf E$,
    \[
\PP(\mathfrak{h}(E)\cap F\neq\emptyset) = \PP((\mathfrak{h}(E)-\mathfrak{h}(a))\cap (F-\mathfrak{h}(a))\neq\emptyset) = \PP(1/\sqrt{2}(\mathfrak{h}(E)-\mathfrak{h}(a))\cap 1/\sqrt{2}(F-\mathfrak{h}(a))\neq\emptyset)\,.
\]
By Corollary, \ref{cor: KPZ fixed point mut abs cont prod}, we have for some probability measure $\eta$ on $\R$,
\[
\eta\otimes \mu \ll \mathrm{ Law\, of } (\mathfrak{h}(a) , \mathfrak{h}(\cdot)-\mathfrak{h}(a))
\]
on $\R\times \mathscr{C}(E; \R)$. Hence, there exist a random variable $Y$ with law $\eta$ and a standard Brownian motion $W$ starting from $(a, 0)$ (mutually independent) such that $(Y, W(\cdot))$ is absolutely continuous with respect to $(\mathfrak{h}(a) , 1/\sqrt{2}(\mathfrak{h}(\cdot)-\mathfrak{h}(a)))$ on $\R\times \mathscr{C}(E; \R)$. 

Thus, to prove 
$\PP(\mathfrak{h}(E)\cap F\neq\emptyset)>0$,
it suffices to prove $\PP(W(E)\cap 1/\sqrt{2}(F-Y)\neq\emptyset)> 0$.
Now, conditionally on $Y$, by \cite[Proposition 1.4]{khoshnevisan2015brownianmotionthermal}, this is true if and only if there exists a probability measure $\sigma$ on $(E-a)\times 1/\sqrt{2}(F-Y)$ such that
\[
\int_{(E-a)\times 1/\sqrt{2}(F-Y)}\int_{(E-a)\times 1/\sqrt{2}(F-Y)}\frac{\mathrm{e}^{-|x-y|^2/(2|t-s|)}}{
|t-s|^{1/2}} \sigma(\diff
s \,\diff x) \sigma(\diff t \,\diff y) > 0\,.
\]
Now, by translation and scaling this is true if and only if there exists some probability measure $\tilde{\sigma}$ on $E\times F$ such that
\[
\int_{E\times F}\int_{E \times F}\frac{\mathrm{e}^{-|x-y|^2/(4|t-s|)}}{
|t-s|^{1/2}} \tilde{\sigma}(\diff
s \,\diff x) \tilde{\sigma}(\diff t \,\diff y) > 0\,,
\]
or equivalently, if and only if $\mathscr{C}_0(E\times F) > 0$.

For the converse implication, suppose that $\mathscr{C}_0(E\times F)=0$, then use the absolute continuity relation from Corollary \ref{cor: KPZ fixed point mut abs cont prod},
$\mathrm{ Law \, of }\; (\mathfrak{h}(a), \mathfrak{h}(\cdot)-\mathfrak{h}(a))\ll \mathrm{Leb}\otimes \mu$
for $a< \inf E$, on $\R\times \mathscr{C}(E; \R)$ and argue entirely analogously, using the characterisation in \cite[Proposition 1.4]{khoshnevisan2015brownianmotionthermal}.
\end{proof}

The condition in Corollary \ref{cor: thermal cap} can also be recast in terms of a geometric condition on the set $E\times F$ involving a certain kind of Hausdorff dimension, which we turn to now. 

Let us define $\varrho$ to be
the \emph{parabolic
metric} on $\R^2$, that is,
\begin{equation*}
\varrho \bigl( (s,x); (t,y) \bigr) \stackrel{\mathrm{def}}{=} \max \bigl( |t-s|^{1/2},
|x-y| \bigr).
\end{equation*}
On the metric space $\mathbf{S} = (\R^2,\varrho)$, also called
\textit{space-time} as we distinguish the spatial and temporal variables, we can define a notion of Hausdorff dimension associated to it. More precisely, for any $\alpha\ge 0$ and $A \subset \R \times \R$, the $\alpha$-dimensional parabolic Hausdorff measure of $A$ is defined by
\[
\mathscr{H}^\alpha(A; \varrho) = \lim_{\delta \to 0} \inf\left\{ \sum_{n=1}^\infty (\mathrm{diam}_\varrho \,U_n)^\alpha : \text{open cover } (U_n)_{n=1}^\infty \text{ of } A \,, \,\sup_{n\ge 1}\, (\mathrm{diam}_\varrho \,U_n) \le \delta \right\}\,.
\]
The parabolic Hausdorff dimension of $A$ is defined by
\[
\mathrm{dim}_{\mathscr{H}}(A; \varrho) \stackrel{\mathrm{def}}{=} \inf\{ \alpha \ge 0 : \mathscr{H}^\alpha(A; \varrho) = 0 \}.
\]

\begin{corollary}[Intersection probabilities]\label{cor: int prob haus dim}
    If $\mathrm{dim}_{\mathscr{H}}(E\times F;\varrho)>1$ then $\PP(\mathfrak{h}(E)\cap F \neq \emptyset ) > 0$. If $\mathrm{dim}_{\mathscr{H}}(E\times F;\varrho)<1$ then $\PP(\mathfrak{h}(E)\cap F\neq \emptyset) = 0$.
\end{corollary}
\begin{proof}
Proceed as in the proof of Corollary \ref{cor: thermal cap} and use Corollary \ref{cor: KPZ fixed point mut abs cont prod} to argue as in page 407 of \cite{khoshnevisan2015brownianmotionthermal} applying the characterisation of the parabolic Hausdorff dimension in terms of thermal capacity.
\end{proof}
In the case where the above probability does not vanish, we are interested in describing the Hausdorff dimension $\mathrm{dim}_{\mathscr{H}}(\mathfrak{h}(E)\cap F)$ of
the random intersection set $\mathfrak{h}(E)\cap F$.
In particular, we seek only to compute
the $L^\infty(\PP)$-norm of that Hausdorff dimension, since in general it is not constant.

Below for $F\subseteq \R$ compact, we denote its Lebesgue measure by $|F|$. We also denote by $\mathrm{dim}_\mathscr{H}$, the Hausdorff dimension on $\R$ with respect to the Euclidean metric (in any dimension).

\begin{corollary}\label{cor: positiveLeb}
If $F \subset\R$  is compact and
$|F|>0$, then
\begin{equation}
\label{eq: posiLeb} \norm{ \mathrm{dim}_{\mathscr{H}} \bigl( \mathfrak{h}(E)\cap F \bigr)}_{L^\infty(\PP)} = 2 \mathrm{dim}_{\mathscr{H}}(E)\land 1.
\end{equation}
If in addition, $\mathrm{dim}_{\mathscr{H}}(E) > 1/2$, then
$\PP(|\mathfrak{h}(E)\cap F|>0)>0$.
\end{corollary}

\begin{proof}
Use \cite[Proposition 1.2]{khoshnevisan2015brownianmotionthermal} to obtain that for a standard Brownian motion starting from $(a, 0)$, $a< \inf E$, with positive probability for any $\varepsilon > 0$,
\[
\PP(\mathrm{dim}_{\mathscr{H}} \bigl( W(E)\cap F \bigr)\ge 2 \mathrm{dim}_{\mathscr{H}}(E)\land 1-\varepsilon)> 0
\]
and almost surely, $\mathrm{dim}_{\mathscr{H}} \bigl( W(E)\cap F \bigr)\le 2 \mathrm{dim}_{\mathscr{H}}(E)\land 1$. Then, argue as in Corollary \ref{cor: thermal cap} and note that the rescaling therein is a bi-Lipschitz homeomorphism, thereby preserving the Hausdorff dimension almost surely.
\end{proof}

The remaining case, and arguably most interesting case,
is when $F$ has Lebesgue measure 0, that is $|F|=0$.
The following result gives a suitable (though
quite complicated) formula that also generalises for Brownian motion in higher dimensions. The proof is entirely analogous to that of Corollary \ref{cor: positiveLeb}, hence omitted.

\begin{corollary}\label{cor: dimh}
If $F \subset\R$ is compact and $|F|=0$, then
\begin{equation}
\label{eq:dimh} \norm{ \mathrm{dim}_{\mathscr{H}} \bigl( \mathfrak{h}(E)\cap F \bigr)}_{L^\infty(\PP)} = \sup \Bigl\{ \gamma>0 : C_\gamma(E \times F) >0 \Bigr\}\,.
\end{equation}
\end{corollary}

\subsection{Geometric properties of the Airy sheet images}

Having established the absolute continuity of additive Brownian motion against the Airy sheet in Theorem \ref{thm: abs cont airy sheet} we compute essential suprema of Hausdorff dimensions of images of compact sets under the Airy sheet and give conditions for the positivity of their Lebesgue measure in terms of the one–dimensional Bessel–Riesz capacity, which we now define. 

\begin{definition}[Bessel-Riesz capacity]\label{def: bessel-riesz cap}
Let $E \subset \R^2$ be compact. The (one-dimensional) \emph{Bessel}-\emph{Riesz} capacity of the set $E\subset \R^2$ is defined by
\begin{equation}\label{def: BR C}
C_{\mathrm{BR}}(E) = \frac{1}{\inf\{ \mathscr{E}_\mathrm{BR}(\mu) : \mu\in \mathscr{P}(E)\}}\in [0, +\infty],
\end{equation}
where $\mathscr{P}(E)$ denotes the set of all Borel probability measures $\mu$ on $E$
and $\mathscr{E}_\mathrm{BR}(\mu)$ is the `energy' of $\mu$ defined by 
\begin{equation*}\label{eq: bessel riesz cap}
\mathscr{E}_\mathrm{BR}(\mu)\stackrel{\mathrm{def}}{=} \int_{E}\int_{E}\frac{1}{
|x-y|^{1/2}} \mu(\diff x) \mu(\diff y)\,.
\end{equation*}
\end{definition}

We now compute the essential supremum of the Euclidean Hausdorff dimension of Airy sheet images using the absolute continuity result in Theorem \ref{thm: abs cont airy sheet} and the spatial regularity of the Airy sheet, \cite[Proposition 10.5]{DOV}. This is the content of the following corollary.

\begin{corollary}\label{cor: haus dim airy sheet image}
     Let $E\subset \R^2$ be bounded Borel and $\mathcal{S}$ an Airy sheet. Then we have
     \[
     \norm{\mathrm{dim}_{\mathscr{H}}(\mathcal{S}(E))}_{L^\infty(\PP)} = 1 \land 2\mathrm{dim}_{\mathscr{H}}(E)\,.
     \]
\end{corollary}
\begin{proof}
    By translation symmetries of the Airy sheet, \eqref{eq: Airy sheet symmetry}, we can assume without loss of generality $E\subset (0, \infty)^2$.
    
    The almost sure upper bound on the random variable $\mathrm{dim}_{\mathscr{H}}(\mathcal{S}(E))$ can be readily established from the H\"{o}lder $1/2-$ regularity of the Airy sheet, \cite[Proposition 10.5]{DOV}, which means the dimension of the image can at most double.

    For the lower bound, \cite[Theorem 6.11, page 181]{xiao2009sample} gives almost surely, $ \mathrm{dim}_{\mathscr{H}}(X(E)) = 1\land 2\mathrm{dim}_{\mathscr{H}}(E)$ where $X$ denotes an additive Brownian motion. Thus, by Theorem \ref{thm: abs cont airy sheet} and translation invariance of Hausdorff dimension, we have with positive probability $ \mathrm{dim}_{\mathscr{H}}(\mathcal{S}(E)) = 1\land  2\mathrm{dim}_{\mathscr{H}}(E)$,
     and so we deduce $\norm{\mathrm{dim}_{\mathscr{H}}(\mathcal{S}(E))}_{L^\infty(\PP)} \ge 1\land 2\mathrm{dim}_{\mathscr{H}}(E)$, which yields the result.
\end{proof}

We now state a condition for when the above random sets have positive Lebesgue measure.

\begin{corollary}\label{cor: bessel-riesz cap pos}
    Let $E\subset \R^2$ be compact and $\mathcal{S}$ an Airy sheet. If $C_{\mathrm{BR}}(E)> 0$, then $\mathbb{P}(\big|\mathcal{S}(E)\big|>0)>0$.
\end{corollary}
\begin{remark}
    By Frostman's characterisation of Hausdorff dimension, if $\mathrm{dim}_{\mathscr{H}}(E) > 1/2$, $\mathbb{P}(\big|\mathcal{S}(E)\big|>0)>0$ and by Corollary \ref{cor: haus dim airy sheet image} if $\mathrm{dim}_{\mathscr{H}}(E) < 1/2$, then $\big|\mathcal{S}(E)\big|=0$ almost surely.
\end{remark}

\begin{proof}
    This, is a direct application of the translation invariance of the Lebesgue measure, the absolute continuity of the Airy sheet (up to centering) against additive Brownian motion, Theorem \ref{thm: abs cont airy sheet} and \cite[Theorem 1.1]{khoshnevisan1998browniansheetimagesbesselriesz} (which is the analogous result of \ref{cor: bessel-riesz cap pos} for additive Brownian motion).
\end{proof}

\section{Appendix}\label{sec: appendix}

Next result shows infinite geodesics in the Airy line ensemble do not `jump instantaneously'.

\begin{lemma}\label{lemma: semi-inf geod jump time pos}
    Fix $x <0$, $1\le k < \ell$. Then, with $\varepsilon^\infty_{\ell, k}$ the last jump time of the almost-surely unique geodesic on the Airy line ensemble from $(x, \ell)$ to $(0, k)$, we have $\varepsilon^\infty_{\ell, k} < 0$ almost surely.
\end{lemma}

\begin{proof}
We show $\PP(\varepsilon^\infty_{\ell, k} = 0)=0$. Expressing last passage values in terms of the Pitman transform see subsection \ref{subsec: pitman trans}, we have
    \[
    \varepsilon^\infty_{\ell, k} = \operatorname*{argmax}_{z\in [x, 0]}(\mathcal{A}[x \to (z, k+1)] + \mathcal{A}[(z, k)\to (0, k)])\,.
    \]
    Now, by the mutual absolute continuity of the centred of the Airy line ensemble with respect to independent Brownian motions, \cite[Theorem 1.1]{dauvergne2024wienerdensitiesairyline}, $\varepsilon^\infty_{\ell, k}$ is mutually absolutely continuous with respect to
    \[
    \operatorname*{argmax}_{z\in [x, 0]}(B[(x, \ell) \to (z, k+1)] + B[(z, k)\to (0, k)]) = \operatorname*{argmax}_{z\in [x, 0]}(B[(x, \ell) \to (z, k+1)] -B_k(z))\,,
    \]
    where $B$ is a family of independent Brownian motions.

    By \cite[Proposition 4.1]{tassopoulos2025inhomogeneousbrownian} and \cite[Theorem 7.1]{tassopoulos2025inhomogeneousbrownian}, the laws of $B[(x, \ell) \to (\cdot, k+1)]$ restricted to $[x/2,0]$ is absolutely continuous with respect to that of a standard Brownian motion starting from $0$ restricted to $[x,0]$. Thus, by independence of $B[(x, \ell) \to (\cdot, k+1)]$ with $B_k$, time-reversal and flip symmetry of Brownian motion,
    \begin{align*}
         \tilde{\varepsilon}^\infty_{\ell, k} &:= \operatorname*{argmax}_{z\in [x/2, 0]}(B[(x, \ell) \to (z, k+1)]-B_k(z))\\
         &= \operatorname*{argmax}_{z\in [x/2, 0]}(B[(x, \ell) \to (z, k+1)]-B[(x, \ell) \to (0, k+1)]-B_k(z))
    \end{align*}
    has the law of the argmax of a Brownian motion on $[x/2, 0]$. Now, by L\'{e}vy's arcsine law, $\tilde{\varepsilon}^\infty_{\ell, k}$ has a density with respect to the Lebesgue measure, which concludes the proof.
\end{proof}

\bibliographystyle{alpha}
\bibliography{refs.bib}

\end{document}